\documentclass[a4paper,10pt]{article}

\usepackage{amsmath,amsthm,amssymb,amsfonts}
\usepackage{geometry}
\usepackage{cleveref}
\usepackage{graphicx}
\usepackage{mathtools}
\usepackage[british]{babel}
\usepackage{float}

\newtheorem{thm}{Theorem}[subsection]
\newtheorem{cor}[thm]{Corollary}
\newtheorem{lem}[thm]{Lemma}
\newtheorem{prop}[thm]{Proposition}

\newtheorem{ques}[thm]{Question}
\newtheorem*{ques*}{Question}

\theoremstyle{remark}
\newtheorem{rem}[thm]{Remark}
\newtheorem{con}[thm]{Convention}
\newtheorem{note}[thm]{Notation}

\theoremstyle{definition}
\newtheorem{defn}[thm]{Definition}
\newtheorem{exm}[thm]{Example}

\title{Product set growth in virtual subgroups of mapping class groups}
\author{Alice Kerr}
\date{}

\begin{document}
	
\maketitle

\begin{abstract}
	We study product set growth in groups with acylindrical actions on quasi-trees and, more generally, hyperbolic spaces. As a consequence, we show that for every surface $S$ of finite type, there exist $\alpha,\beta>0$ such that for any finite symmetric subset $U$ of the mapping class group $MCG(S)$ we have $|U^n|\geqslant (\alpha|U|)^{\beta n}$, so long as no finite index subgroup of $\langle U\rangle$ has an infinite centre. This gives us a dichotomy for the finitely generated subgroups of mapping class groups, which extends to virtual subgroups.
	
	As right-angled Artin groups embed as subgroups of mapping class groups, this result applies to them, and so also applies to finitely generated virtually special groups. We separately prove that we can quickly generate loxodromic elements in right-angled Artin groups, which by a result of Fujiwara \cite{Fujiwara2021} shows that the set of growth rates for many of their subgroups are well-ordered.
\end{abstract}

\section{Introduction}

For a finite subset $U$ of a group $G$, we define its \emph{$n$th product set} to be
\begin{equation*}
U^n=\{u_1\cdots u_n\ :\ u_1,\ldots,u_n\in U\}
\end{equation*}
The study of growth in groups is the study of how $|U^n|$ behaves as $n$ varies. For infinite groups, the classical notion of group growth is when $G$ is taken to be finitely generated, and $U$ is taken to be a ball in $G$ with respect to the word metric induced by some finite generating set $S$. The usual question that is asked is whether $|U^n|$ grows like a polynomial function as $n\to\infty$, an exponential function, or if it lies somewhere between the two. This is known to have links to the algebraic properties of the group, the most famous example being Gromov's proof that the groups of polynomial growth are exactly the virtually nilpotent groups \cite{Gromov1981}.

In finite groups, on the other hand, there is little interest in estimating $|U^n|$ for large powers of $n$, as this is bounded above by the size of the group, and so will eventually become constant. Instead, questions about group growth in this setting tend to focus on getting precise bounds on low powers, such as $|U^2|$ or $|U^3|$.

The question we are interested in here combines aspects of these two cases. Our focus will be on infinite groups, and for $U$ a general finite subset we would like to estimate $|U^n|$ for every $n\in\mathbb{N}$. On the other hand, instead of only being concerned with the asymptotic profile of $|U^n|$, we will be interested in getting an explicit lower bound on $|U^n|$, with this lower bound dependent on the size of $U$.

More specifically, we will be interested in when there exist $\alpha,\beta>0$ such that a finite subset $U$ satisfies $|U^n|\geqslant (\alpha|U|)^{\beta n}$ for every $n\in\mathbb{N}$. If all of the finite (symmetric) generating sets of an infinite group satisfy this inequality for the same constants $\alpha$ and $\beta$, then we say that the group has \emph{uniform product set growth}. The aim here, however, is not just to prove this property for the group itself, but to find a dichotomy for its finitely generated subgroups.

%

\begin{ques}
	\label{MainQuestion}
	For a group $G$, does there exists a class of subgroups $\mathcal{H}$ such that the following hold?
	\begin{itemize}
		\item Every finitely generated subgroup $H \notin\mathcal{H}$ has uniform product set growth, for the same $\alpha,\beta>0$.
		\item No finitely generated subgroup $H \in\mathcal{H}$ has uniform product set growth, for any $\alpha,\beta>0$.
	\end{itemize}
\end{ques}

This question has been completely answered in the case of free groups \cite{Safin2011}, and more generally hyperbolic groups \cite{Delzant2020}, with $\mathcal{H}$ being the virtually cyclic groups. In these cases the growth was shown for all finite subsets generating non-virtually cyclic groups, rather than restricting to the symmetric subsets. A similar dichotomy for the non-elementary subgroups of relative hyperbolic groups was shown in \cite{Cui2021}. This was recently expanded to the non-symmetric case in \cite{Wan2023}, along with several other product set growth results, including a dichotomy for subgroups of fundamental groups of Riemannian manifolds with pinched negative curvature. For a more complete history of such results, see Section 3.4 of \cite{Kerr2022}.

The main result of this paper is that we can answer \Cref{MainQuestion} for virtual subgroups of mapping class groups.

\begin{thm}[\Cref{MainResult}]
	\label{IntroMainResult}
	Let $MCG(S)$ be a mapping class group, and let $G$ be a group which virtually embeds into $MCG(S)$. There exist $\alpha,\beta>0$ such that for every finite symmetric $U\subset G$, at least one of the following must hold:
	\begin{enumerate}
		\item $\langle U\rangle$ has a finite index subgroup with infinite centre.
		\item $|U^n| \geqslant (\alpha|U|)^{\beta n}$ for every $n \in \mathbb{N}$. 
	\end{enumerate}
\end{thm}

\begin{rem}
	Here we use the definition of mapping class groups where punctures in the surface are allowed to permute, but boundary components are fixed pointwise. A brief discussion of the variation in definitions of the mapping class group is given in \Cref{NotAH}.
\end{rem}

Answering our question completely for a group also naturally answers it for all of its subgroups. One class of groups that is known to embed in mapping class groups are the right-angled Artin groups (see \cite{Clay2012}, for example). In particular, this means that we can answer \Cref{MainQuestion} for finitely generated virtually special groups \cite{Haglund2008}, which includes the finitely generated Coxeter groups \cite{Haglund2010}.

\begin{thm}[\Cref{VirtualRaag}]
	\label{IntroVirtualRaag}
	Let $\Gamma$ be a finite graph, and $A(\Gamma)$ the right-angled Artin group associated to $\Gamma$. Let $G$ be a group which virtually embeds into $A(\Gamma)$. There exist constants $\alpha,\beta > 0$ such that for every finite symmetric $U \subset G$, at least one of the following must hold:
	\begin{enumerate}
		\item $\langle U\rangle$ has a finite index subgroup with infinite centre.
		\item $|U^n| \geqslant (\alpha|U|)^{\beta n}$ for every $n \in \mathbb{N}$.
	\end{enumerate}
\end{thm}

The following proposition tells us that the subgroups in the first cases of \Cref{IntroMainResult} and \Cref{IntroVirtualRaag} cannot have uniform product set growth for any $\alpha$ and $\beta$, which justifies that these results do indeed give a dichotomy for the finitely generated subgroups.

\begin{prop}[\Cref{InfFICentre}]
	\label{IntroFICentre}
	Let $G$ be a finitely generated group, and suppose that $G$ has a finite index subgroup $H$ with infinite centre. Then for any constants $\alpha,\beta>0$, we can find a finite generating set $U$ of $G$ such that $|U^n|< (\alpha|U|)^{\beta n}$ for some $n\in\mathbb{N}$.
\end{prop}

The intuitive reason as to why such groups cannot have uniform product set growth is easier to see in the case that $G$ itself has infinite centre. In this case, the idea is that we can add arbitrarily many central elements to any generating set $U$. As these elements are contained in an abelian subgroup, which does not have exponential growth, this has minimal impact on the long term growth of $U^n$. This means that we cannot link the growth of $U^n$ with the size of $U$ in a meaningful way. As a consequence, if our group contains subgroups with exponential growth but infinite centres, any dichotomy of subgroups we get in answer to \Cref{MainQuestion} is likely to be different to the dichotomy given by the Tits alternative, if such an alternative exists for the group in question. For example, the dichotomy in \Cref{IntroMainResult} is not the same as the one given by the Tits alternative for mapping class groups \cite{Ivanov1984,McCarthy1985}.


\Cref{MainQuestion} has strong links to the more commonly studied notion of uniform exponential growth. Specifically, any finitely generated group which has uniform product set growth will necessarily have uniform exponential growth. Moreover, the finitely generated subgroups that do not lie in $\mathcal{H}$ will have uniform uniform exponential growth, in the sense that they will have a common lower bound on their exponential growth rates. Such properties have already been shown for a wide range of groups, including mapping class groups and right-angled Artin groups \cite{Mangahas2010}, although the question of uniform exponential growth is still open more generally for acylindrically hyperbolic groups.

On the other hand, \Cref{IntroFICentre} tells us that uniform product set growth is a strictly stronger property than uniform exponential growth. If $H$ has uniform exponential growth then so does $H\times\mathbb{Z}$, however \Cref{IntroFICentre} tells us that this product does not have uniform product set growth. The key difference is that, as mentioned above, uniform product set growth asks for the lower bound on the growth to be given in terms of the size of the set, whereas uniform exponential growth just asks for some common lower bound. Linking growth rates with the size of the generating set in question is a key step in proofs about growth rates being well-ordered, see \cite{Fujiwara2020,Fujiwara2021} and \Cref{IntroWellOrderedCor}.

A positive answer to \Cref{MainQuestion} has consequences for other types of product set growth as well, namely Helfgott type growth, as observed by Jack Button \cite{Button2013}. Uniform product set growth also has links with approximate groups, as any approximate group that satisfies the relevant inequality must have bounded size. More details on these applications are given in Section 2.1.

\begin{prop}[\Cref{Approx2}]
	\label{IntroApprox2}
	Let $G$ be a group, and let $k\geqslant 1$, $\alpha,\beta>0$. Suppose $U$ is a $k$-approximate group in $G$ satisfying $|U^n|\geqslant (\alpha|U|)^{\beta n}$ for every $n\in\mathbb{N}$. Then $|U|\leqslant \frac{(2k^2)^{\frac{6}{\beta}}}{\alpha^2}$.
\end{prop}

As previously mentioned, Delzant and Steenbock were able to answer \Cref{MainQuestion} for hyperbolic groups \cite{Delzant2020}. Alongside this, they were also able to generalise their methods to groups acting acylindrically on trees and hyperbolic spaces. The subgroup structure of an acylindrically hyperbolic group may in general be much more complicated than the subgroup structure of a hyperbolic group, so it is not usually sufficient to simply rule out the virtually cyclic subgroups when looking for uniform product set growth, as we can see in \Cref{IntroMainResult}. In particular, it is hard to say anything about subsets whose image under the orbit map have small diameter.

Delzant and Steenbock ruled out such subsets by using a certain notion of displacement. If $(X,d)$ is a hyperbolic space that a group $G$ acts on by isometries, and $U\subset G$ is finite, then we pick $x_0\in X$ such that $\sum_{u\in U}d(x_0,ux_0)$ is (almost) minimised. The \emph{displacement} of $U$ is $\lambda_0(U)=\max_{u\in U}d(x_0,ux_0)$. For a more detailed definition, see \Cref{Displacement}. An alternative approach can be found in \cite{Coulon2022}.

Delzant and Steenbock's result for acylindrical actions on hyperbolic spaces included logarithm terms that were not present in their other results, due to the use of Gromov's tree approximation lemma. In \cite{Kerr2023}, it was shown that tree approximation is uniform if the space in question is a quasi-tree, rather than a hyperbolic space. We will use this to improve Delzant and Steenbock's bounds in the quasi-tree case by removing these logarithm terms. This generalises their result for actions on trees.

\begin{thm}[\Cref{KChoice}]
	\label{IntroKChoice}
	Let $G$ be a group acting acylindrically on a quasi-tree. Then there exist constants $K>0$ and $\alpha>0$ such that for every finite $U\subset G$, at least one of the following must hold:
	\begin{enumerate}
		\item $\langle U\rangle$ is virtually $\mathbb{Z}$.
		\item $\lambda_0(U)< K$.
		\item $|U^n|\geqslant (\alpha|U|)^{\lfloor\frac{n+1}{2}\rfloor}$ for every $n\in\mathbb{N}$.
	\end{enumerate}
\end{thm}

\begin{rem}
	By an observation of Safin, $\lfloor\frac{n+1}{2}\rfloor$ is the best possible exponent for a lower bound on product set growth in free groups, and so this will also be the best possible exponent for any wider class of groups.
\end{rem}

A result of Balasubramanya tells us that every acylindrically hyperbolic group admits some acylindrical action on a quasi-tree \cite{Balasubramanya2017}, so the above theorem applies to all such groups. The class of acylindrically hyperbolic groups is a wide one, including the hyperbolic groups and relatively hyperbolic groups that are not virtually cyclic, as well as many cubical groups, right-angled Artin groups, and most mapping class groups of surfaces without boundary \cite{Osin2016}. The above statement is also interesting because it is not restricted to symmetric sets, and instead holds for all finite subsets.

We will want to apply \Cref{IntroKChoice} to investigate uniform product set growth for subgroups of acylindrically hyperbolic groups, however we cannot do this immediately, as the displacement condition is a condition on sets, rather than on the subgroups they generate. In Section 3.2 we will show that one possible way of overcoming this is to find loxodromic elements in the action.

\begin{prop}[\Cref{ShortLoxodromicCor}]
	\label{IntroShortLoxodromic}
	Let $G$ be a group acting acylindrically on a quasi-tree. There exist $\alpha,\beta>0$ such that for every finite $U\subset G$ such that $U^k$ contains a loxodromic element for some $k\in\mathbb{N}$, and $\langle U\rangle$ is not virtually $\mathbb{Z}$, we have that $|U^n|\geqslant (\alpha|U|)^{\frac{\beta n}{k}}$ for every $n\in\mathbb{N}$.
\end{prop}

In other words, if we can find loxodromic elements uniformly quickly in the generating sets of a subgroup, then we will have the uniform product set growth that we are looking for. Fujiwara has recently shown a very similar result for actions on hyperbolic spaces, using an alternative method of generating free subgroups of large enough rank \cite{Fujiwara2021}. The downside to this result compared to \Cref{IntroShortLoxodromic} is that it is restricted to symmetric sets, however it is very useful in the cases where no well-understood acylindrical action on a quasi-tree exists, such as in the proof of \Cref{IntroMainResult}.

Another useful tool in the proof of \Cref{IntroMainResult} is the fact that we already know that we can generate loxodromic elements uniformly quickly in mapping class groups. In the following result we can think of finding an element with the same active subsurface as our subgroup as being roughly the same as finding a loxodromic element in the given subgroup.

\begin{thm}
	\emph{\cite{Mangahas2013}}
	\label{IntroShortPA}
	Consider a mapping class group $MCG(S)$, where $S$ is a non-sporadic connected surface without boundary. There exists a constant $N=N(S)\in\mathbb{N}$ such that for any finite symmetric $U\subset MCG(S)$, there exists $n\leqslant N$ and $f\in U^n$ such that $f$ has the same active subsurface as $\langle U\rangle$.
\end{thm}

As mentioned previously, the result for right-angled Artin groups follows directly from the fact that they embed as subgroups of mapping class groups. It is however also interesting to consider the right-angled Artin group case separately, as in doing so we prove an analogous result to \Cref{IntroShortPA} for right-angled Artin groups. In other words, we can show that we can quickly generate elements of full support in our subgroup, which allows us to quickly find loxodromic elements in the action of the subgroup on the associated extension graph. In the case considered in \Cref{IntroExtensionShortLoxPrime}, this extension graph will be a quasi-tree, which allows us to apply \Cref{IntroShortLoxodromic}.

\begin{note}
	For a finite graph $\Gamma=(V,E)$ we have that $A(\Gamma)$ is the associated right-angled Artin group, $\Gamma^e$ is the associated extension graph, and for a subset $V'\subset V$ we have that $\Gamma(V')$ is the induced subgraph on those vertices. Given $U\subset A(\Gamma)$, we say that $\text{esupp}(U)\subset V$ is the essential support of $U$. For more detailed definitions see Section 4.1.
\end{note}

\begin{thm}[\Cref{FullSupport}]
	\label{IntroFullSupport}
	Let $\Gamma$ be a finite graph. There exists a constant  $N=N(\Gamma)\in\mathbb{N}$ such that for every finite symmetric $U\subset A(\Gamma)$, there exists $n\leqslant N$ and $g\in U^n$ such that $\emph{esupp}(g)=\emph{esupp}(U)$.
\end{thm}

\begin{cor}[\Cref{ExtensionShortLoxPrime}]
	\label{IntroExtensionShortLoxPrime}
	Let $\Gamma$ be a finite graph. There exists a constant $N=N(\Gamma)\in\mathbb{N}$ such that for every finite symmetric $U\subset A(\Gamma)$, where $\Gamma(\emph{esupp}(U))$ is connected, and is neither an isolated vertex nor a join, there exists $n\leqslant N$ such that $U^n$ contains a loxodromic element on $\Gamma(\emph{esupp}(U))^e$.
\end{cor}

A recent result of Fujiwara states that if a group $G$ has an acylindrical and non-elementary action on a hyperbolic graph, and the product of any finite generating set of $G$ contains a loxodromic element within a bounded power, then the set of exponential growth rates of $G$ (with respect to its finite generating sets) will be well-ordered so long as $G$ is equationally Noetherian \cite{Fujiwara2021}. Linear groups are equationally Noetherian \cite{Baumslag1999}, and right-angled Artin groups are linear \cite{Hsu1999}, so we are able to use \Cref{IntroExtensionShortLoxPrime} to obtain the following result.

\begin{thm}[\Cref{WellOrderedCor}]
	\label{IntroWellOrderedCor}
	Let $\Gamma$ be a finite graph. Suppose that $G\leqslant A(\Gamma)$ is finitely generated, and is not contained non-trivially in a direct product where the projection of $G$ to more than one factor has exponential growth. Then the set of exponential growth rates of $G$ (with respect to its finite generating sets) is well-ordered.
\end{thm}

\textbf{Structure of the paper:} In Section 2, we give a few of the applications of product set growth, as well as some fundamental results about which groups can and cannot have uniform product set growth. In Section 3, we combine uniform tree approximation for quasi-trees from \cite{Kerr2023} with Delzant and Steenbock’s methods in \cite{Delzant2020}, which allows us to get new results for product set growth in acylindrically hyperbolic groups. In Section 4, we apply the results of the previous sections to answer \Cref{MainQuestion} for virtual subgroups of mapping class groups. We also prove a result about quickly generating loxodromic elements in right-angled Artin groups.

\textbf{Acknowledgements:} The author would like to thank Cornelia Dru\c{t}u for the many questions, suggestions, and discussions which helped shape this work, and for taking the time to read multiple drafts of this paper. The author is also grateful to Thomas Delzant for pointing out the relevance of loxodromic elements in the context of applying \Cref{IntroKChoice}, and to both him and Emmanuel Breuillard for helping to improve this paper while acting as the author’s thesis examiners. Thanks also go to many others for their generosity in sharing their time and knowledge, and for giving valuable comments and feedback on earlier drafts, in particular Ric Wade, Jack Button, Johanna Mangahas, Sahana Balasubramanya, Nicolaus Heuer, Elia Fioravanti, Ashot Minasyan, Alessandro Sisto, and Markus Steenbock. Special thanks go to the referee of this paper, for many helpful comments and corrections, some of which allowed for improvements to the main theorems.
	
\section{Product set growth}

Recall that for a finite subset $U$ of a group $G$, its $n$th product set is $U^n=\{u_1\cdots u_n\ :\ u_1,\ldots,u_n\in U\}$. We are interested in estimating the size of $U^n$ as $n$ increases, specifically in the case that $G$ is an infinite group. As mentioned in the introduction, when considering infinite groups the most commonly studied form of group growth is the asymptotic growth of balls in a finitely generated group.

\begin{defn}
	Let $G$ be a finitely generated group, and let $S$ be a finite generating set of $G$. Let $B_S(n)$ denote the ball of radius $n$ under the word metric induced by $S$, centred at the identity $e\in G$. The \emph{exponential growth rate of $G$ with respect to $S$} is
	\begin{equation*}
	\omega(G,S)=\lim_{n\to\infty}|B_S(n)|^{\frac{1}{n}}.
	\end{equation*}
	If $\omega(G,S)>1$ for some (equivalently any) finite generating set $S$, then $G$ has \emph{exponential growth}.
\end{defn}

\begin{rem}
	In the notation of product sets, it is easy to note that when $U=B_S(1)$ we will have that $U^n=B_S(n)$.
\end{rem}

The fact that $\omega(G,S)>1$ for some finite generating set if and only if it is true for every generating set is a standard result, and is not hard to check. We therefore only need to prove this inequality for a single generating set, so in many cases exponential growth can be a relatively easy property to prove. On the other hand, although this tells us that all finite generating sets will exhibit some kind of exponential growth, there is not necessarily any control over their rates of growth, which may be arbitrarily close to 1. If we wish to have this type of control, then we will need to prove something stronger.

\begin{defn}
	Let $G$ be a finitely generated group, and let
	\begin{equation*}
	\mathcal{S}=\{S\subset G\ :\ S\text{ is a finite generating set of }G\}.
	\end{equation*}
	The \emph{exponential growth rate of $G$} is
	\begin{equation*}
	\omega(G)=\inf_{S\in \mathcal{S}}\omega(G,S).
	\end{equation*}
	If $\omega(G)>1$ then $G$ has \emph{uniform exponential growth}.
\end{defn}

Clearly uniform exponential growth implies exponential growth. Most groups that are known to have exponential growth also have uniform exponential growth, including free groups, hyperbolic groups, and relatively hyperbolic groups \cite{Xie2007}. Away from non-positively curved groups, it is known that every elementary amenable group that has exponential growth also has uniform exponential growth \cite{Osin2004}, and the same is true for any finitely generated subgroup of $\text{GL}_n(\mathbb{C})$ \cite{Eskin2005}. This is not true in general, however \cite{Wilson2004}, and there are other cases where this equivalence is still open, such as for acylindrically hyperbolic groups.

If we do understand the growth of a group, then it makes sense to investigate the growth of its subgroups. One possible question is to ask if there is a common lower bound on the exponential growth rates of the subgroups with uniform exponential growth.

\begin{defn}
	Let $G$ be a group, and let $\mathcal{H}$ be a collection of finitely generated subgroups of $G$. We say that the collection $\mathcal{H}$ has \emph{uniform uniform exponential growth} if $\inf_{H\in\mathcal{H}}\omega(H)>1$.
\end{defn}

For most of this paper, however, we will not just be considering the growth of balls, but rather the case that $U$ is a general finite subset of our group. We also recall from the introduction that we are interested in getting a specific lower bound on $|U^n|$, where this lower bound is dependent on the size of $U$.

\begin{defn}
	Let $G$ be an infinite finitely generated group. We say that $G$ has \emph{uniform product set growth} if there exist $\alpha,\beta>0$ such that for every finite (symmetric) generating set $U$ of $G$, we have that $|U^n|\geqslant (\alpha|U|)^{\beta n}$ for every $n\in\mathbb{N}$.
\end{defn}

\begin{rem}
	If the group $G$ in question was finite, then this would be trivially satisfied by letting $\alpha=\frac{1}{|G|}$. For this reason we will generally only be interested in determining this property for infinite groups.
\end{rem}

The aim of \Cref{MainQuestion} is to find out which of the finitely generated subgroups of a group $G$ have uniform product set growth, ideally for the same $\alpha$ and $\beta$. In Section 2.1 we will prove some consequences of a group or subgroup having this property. Section 2.2 will then be spent discussing some cases in which we can determine whether a subgroup has uniform product set growth or not. This will be important in Section 4, where we will answer \Cref{MainQuestion} for virtual subgroups of mapping class groups.

\subsection{Applications of uniform product set growth}


In this section we will give some properties that are implied by uniform product set growth.

\subsubsection{Uniform exponential growth}

Here we check that uniform product set growth implies uniform exponential growth, even when using the weaker form where we only ask that the inequality holds for symmetric generating sets.

\begin{lem}
	\label{UEG}
	Let $G$ be a finitely generated group. Suppose there exist constants $\alpha,\beta>0$ such that for every finite symmetric generating set $U$ of $G$, we have that $|U^n|\geqslant (\alpha|U|)^{\beta n}$ for every $n\in\mathbb{N}$. Then either $G$ is finite with $|G|\leqslant \frac{1}{\alpha}$, or $G$ has uniform exponential growth.
\end{lem}

\begin{proof}
	Suppose that $G$ is finite. Then $G$ is a finite generating set of itself, and $G^n=G$ for every $n\in\mathbb{N}$. Therefore $|G|\geqslant (\alpha|G|)^{\beta n}$ for every $n\in\mathbb{N}$, so $|G|\leqslant \frac{1}{\alpha}$.
	
	Now suppose that $G$ is infinite, and let $S$ be a finite generating set of $G$. We want to calculate a lower bound on $\omega(G,S)=\lim_{n\to\infty}\frac{1}{n}\log(|B_S(n)|)$. We begin by noting that $B_S(n)=B_S(1)^n$, so for every $n\in\mathbb{N}$ we have that $|B_S(n)|\geqslant(\alpha|B_S(1)|)^{\beta n}$. If $|B_S(1)|\leqslant \frac{1}{\alpha}$ then our equation would tell us nothing about $|B_S(n)|$, however to calculate the limit it is sufficient to consider a subsequence, so we just have to start this subsequence with a large enough set.
	
	For any $m\in\mathbb{N}$ we also have that $B_S(m)$ is also a finite generating set of $G$, and $B_S(m)^n=B_S(mn)$ for all $n\in\mathbb{N}$, so $|B_S(mn)|\geqslant(\alpha|B_S(m)|)^{\beta n}$. We can also see that $|B_S(m)|\geqslant m-1+|S|$, so choosing $m=\lceil\frac{1}{\alpha}\rceil+1$ we get that
	\begin{equation*}
	\alpha|B_S(m)|\geqslant \alpha(\lceil \frac{1}{\alpha}\rceil+|S|)\geqslant 1+\alpha|S|>1.
	\end{equation*}
	With this choice of $m$ we now get that
	\begin{align*}
	|B_S(mn)|^{\frac{1}{mn}} & \geqslant (\alpha|B_S(m)|)^{\frac{\beta}{m}}
	\\ & \geqslant(1+\alpha|S|)^{\frac{\beta}{\lceil1/\alpha\rceil+1}}.
	\end{align*}
	Hence we have that
	\begin{equation*}
	\omega(G,S)=\lim_{n\to\infty}|B_S(n)|^{\frac{1}{n}}=\lim_{n\to\infty}|B_S(mn)|^{\frac{1}{mn}}\geqslant(1+\alpha|S|)^{\frac{\beta}{\lceil1/\alpha\rceil+1}},
	\end{equation*}
	so as $S$ was arbitrary, and $|S|\geqslant 1$, we get that
	\begin{equation*}
	\omega(G)=\inf_{S\in\mathcal{S}}\omega(G,S)\geqslant(1+\alpha)^{\frac{\beta}{\lceil1/\alpha\rceil+1}}>1.
	\end{equation*}
\end{proof}

\begin{rem}
	If some finite generating set $U$ of an infinite group satisfies $|U^n|\geqslant (\alpha|U|)^{\beta n}$ for every $n\in\mathbb{N}$, then this is enough to imply that the group has exponential growth. A set $U$ contained in a group with subexponential growth can only satisfy this inequality in the trivial case that $|U|\leqslant \frac{1}{\alpha}$.
\end{rem}

A natural corollary of this is that if we have a positive answer to \Cref{MainQuestion}, that is we have uniform product set growth for a collection of infinite subgroups with the same $\alpha$ and $\beta$, then that collection has uniform uniform exponential growth.

In the other direction, it is not hard to see that uniform product set growth is actually a strictly stronger property than uniform exponential growth. An example of this difference will be given in Section 2.2, where we consider a direct product where one of the factors has uniform exponential growth and another does not.

\subsubsection{Helfgott type growth}

As mentioned in the introduction, the study of product set growth for arbitrary sets has often previously focussed on estimating the size of double and triple products of those sets. Here we give a link between uniform product set growth and the growth of triple products.

\begin{defn}
	\cite{Button2013}
	An infinite group $G$ has \emph{Helfgott type growth} if there exist $c,\delta>0$ such that any finite $U\subset G$, where $\langle U\rangle$ is not virtually nilpotent, satisfies $|U^3|\geqslant c|U|^{1+\delta}$.
\end{defn}

It is known that if the size of the triple product is bounded, then the size of every higher product can also be bounded. The following is not explicitly stated in \cite{Helfgott2008}, however it can be directly obtained from the work there.

\begin{lem}
	\label{SmallTripling}
	\emph{\cite{Helfgott2008}}
	Let $G$ be a group, and let $U\subset G$ be finite. If $|U^3|\leqslant K|U|$ for some $K> 0$, then $|U^n|\leqslant (8K^3)^{n-2}|U|$ for every $n\geqslant 3$.
\end{lem}

\begin{proof}
	Let $\overline{U}=U\cup U^{-1}$. It is shown in \cite{Helfgott2008} that
	\begin{equation*}
	\frac{|\overline{U}^n|}{|U|}\leqslant\bigg(\frac{|\overline{U}^3|}{|U|}\bigg)^{n-2}
	\end{equation*}
	for $n\geqslant 3$. Clearly $|U^n|\leqslant |\overline{U}^n|$, so it remains to bound $|\overline{U}^3|/|U|$ above by a function of $|U^3|/|U|$. By considering all possible terms in $\overline{U}^3$, we can see that
	\begin{equation*}
	|\overline{U}^3|\leqslant 2(|UUU|+|UUU^{-1}|+|UU^{-1}U|+|U^{-1}UU|).
	\end{equation*}
	Here we use that $UU^{-1}U^{-1}=(UUU^{-1})^{-1}$, so they have the same cardinality. Similar logic applies to the other products. In \cite{Helfgott2008} it is shown that the cardinality of each of these products is bounded above by $|U|(|U^3|/|U|)^3$, and therefore
	\begin{equation*}
	\frac{|\overline{U}^3|}{|U|}\leqslant 8\bigg(\frac{|U^3|}{|U|}\bigg)^3.
	\end{equation*}
	Combining these inequalities, we see that
	\begin{equation*}
	\frac{|U^n|}{|U|}\leqslant\bigg(8\bigg(\frac{|U^3|}{|U|}\bigg)^3\bigg)^{n-2}.
	\end{equation*}
	If $|U^3|\leqslant K|U|$ for some $K\geqslant 1$, then substituting this into our above inequality gives us the result that we were looking for.
\end{proof}

When uniform product set growth is given in the form $|U^n|\geqslant (\alpha|U|)^{\lfloor\frac{n+1}{2}\rfloor}$ for every $n\in\mathbb{N}$, we automatically get Helfgott type growth by simply taking $n=3$, $c=\alpha^2$, and $\delta=1$. However \Cref{SmallTripling} allows us to also obtain Helfgott type growth in the case of the more general type of uniform product set growth, that is when $|U^n|\geqslant (\alpha|U|)^{\beta n}$ for every $n\in\mathbb{N}$. This observation, and the following proof, were given by Jack Button.

\begin{prop}
	\label{SmallTripling2}
	Let $G$ be a group, and let $U\subset G$ be finite. If there exist $\alpha,\beta>0$ such that $|U^n|\geqslant (\alpha|U|)^{\beta n}$ for every $n\in\mathbb{N}$, then $|U^3|\geqslant (\frac{\alpha^{\beta}}{8})^{\frac{1}{3}}|U|^{1+\frac{\beta}{6}}$.
\end{prop}

\begin{proof}
	This is trivially true if $|U|=1$, as in this case our inequality tells us that $\alpha^{\beta}\leqslant 1$. We therefore assume that $|U|\geqslant 2$, and suppose that $|U^n|\geqslant (\alpha|U|)^{\beta n}$ for every $n\in\mathbb{N}$, but $|U^3|< (\frac{\alpha^{\beta}}{8})^{\frac{1}{3}}|U|^{1+\frac{\beta}{6}}$. We can now take $(\frac{\alpha^{\beta}}{8})^{\frac{1}{3}}|U|^{\frac{\beta}{6}}$ to be our $K$ from \Cref{SmallTripling}, and obtain
	\begin{equation*}
		(\alpha|U|)^{\beta n}\leqslant |U^n|\leqslant (8K^3)^{n-2}|U|=(\alpha^{\beta}|U|^{\frac{\beta}{2}})^{n-2}|U|=\frac{(\alpha^{\beta}|U|^{\frac{\beta}{2}})^n|U|}{(\alpha^{\beta}|U|^{\frac{\beta}{2}})^2}
	\end{equation*}
	for every $n\geqslant 3$. We therefore have that
	\begin{equation*}
		|U|^{\frac{\beta n}{2}}\leqslant\frac{|U|}{\alpha^{2\beta}|U|^{\beta}}
	\end{equation*}
	for every $n\geqslant 3$. However, as $|U|\geqslant 2$, we also have that $|U|^{\frac{\beta n}{2}}$ tends to infinity. This contradicts the existence of a fixed upper bound, as given in the above inequality, and hence we must have that $|U^3|\geqslant (\frac{\alpha^{\beta}}{8})^{\frac{1}{3}}|U|^{1+\frac{\beta}{6}}$.
\end{proof}

\subsubsection{Approximate groups}

One other possible motivation for studying product set growth is its link with approximate groups, as defined in \cite{Tao2008}.

\begin{defn}
	Let $G$ be a group, and let $k\geqslant 1$. A finite symmetric subset $U\subset G$ is a \emph{$k$-approximate group} if there exists a finite subset $X\subset G$ such that $|X|\leqslant k$ and $U^2\subset XU$.
\end{defn}

It is standard that approximate groups have small tripling, which led Button to observe that Helfgott type growth can be used to bound the size of approximate groups \cite{Button2013}. Using the same reasoning as Button, we can use \Cref{SmallTripling2} to obtain a bound on the size of any approximate groups which satisfy the inequality in the definition of uniform product set growth.

\begin{prop}
	\label{Approx2}
	Let $G$ be a group, and let $k\geqslant 1$, $\alpha,\beta>0$. Suppose $U$ is a $k$-approximate group in $G$ satisfying $|U^n|\geqslant (\alpha|U|)^{\beta n}$ for every $n\in\mathbb{N}$. Then $|U|\leqslant \frac{(2k^2)^{\frac{6}{\beta}}}{\alpha^2}$.
\end{prop}

\begin{proof}
	Let $X\subset G$ be such that $|X|\leqslant k$ and $U^2\subset XU$. Then $U^3\subset XU^2\subset X^2 U$, so 
	\begin{equation*}
		|U^3|\leqslant |X^2U|\leqslant |X|^2|U|\leqslant k^2|U|.
	\end{equation*}
	By \Cref{SmallTripling2} we have that $|U^3|\geqslant (\frac{\alpha^{\beta}}{8})^{\frac{1}{3}}|U|^{1+\frac{\beta}{6}}$. We therefore have that
	\begin{align*}
		\bigg(\frac{\alpha^{\beta}}{8}\bigg)^{\frac{1}{3}}|U|^{1+\frac{\beta}{6}}\leqslant k^2|U| & \implies |U|^{\frac{\beta}{6}}\leqslant k^2\bigg(\frac{8}{\alpha^\beta}\bigg)^{\frac{1}{3}}
		\\ & \implies |U|\leqslant \frac{(2k^2)^{\frac{6}{\beta}}}{\alpha^2}.
	\end{align*}
\end{proof}



\subsection{Obstructions and sufficient conditions}

Here we will prove a collection of results about when certain classes of subgroups can or cannot have uniform product set growth. We will begin by showing that any group with a finite index subgroup with infinite centre cannot have uniform product set growth. Our focus will then switch to subgroups of direct products, as there are cases where the uniform product set growth property can pass from the factors to the subgroup, and cases where certain factors make uniform product set growth impossible. We will also show that uniform product set growth passes to finite index supergroups. These tools will be vital in Section 4, when attempting to answer \Cref{MainQuestion} for virtual subgroups of right-angled Artin groups and mapping class groups.

\subsubsection{Infinite order central subgroups}

So far the only subgroups we have completely ruled out when attempting to find those with uniform product set growth have been those that do not have uniform exponential growth (see \Cref{UEG}). This may be enough to get us a dichotomy in certain cases, however it will not be sufficient in general.

%
%
%
%
%

Our main example of this will be when our group has a finite index subgroup with infinite centre, although we can in fact show a slightly more general result. We first consider any infinite normal subgroup without exponential growth, where this subgroup is a union of finite conjugacy classes. The obstruction given by central subgroups was initially pointed out by Jack Button, and the extension below came from discussions with Ric Wade.

\begin{prop}
	\label{InfNormal}
	Let $G$ be a finitely generated group, and suppose that it has an infinite normal subgroup $H$, where no finitely generated subgroup of $H$ has exponential growth, and for every $h\in H$ the conjugacy class of $h$ in $G$ is finite. Then for any constants $\alpha,\beta>0$ we can find a finite generating set $U$ of $G$ such that $|U^n|< (\alpha|U|)^{\beta n}$ for some $n\in\mathbb{N}$.
\end{prop}

\begin{proof}
	Let $G$ and $H$ be as above, and note that $G$ must be infinite as it has an infinite subgroup. Fix constants $\alpha,\beta>0$, and let $U$ be some finite generating set of $G$. As $H$ is infinite, we can pick a finite set $V'\subset H$ such that $|V'|> \frac{1}{\alpha}|U|^{\frac{1}{\beta}}$. Let $V$ be the union of conjugacy classes $\bigcup_{v\in V'}\{gvg^{-1}:g\in G\}$. As $H$ is a normal subgroup of $G$, and conjugacy classes are finite, we have that $V$ is a finite subset of $H$, and $|V|\geqslant |V'|> \frac{1}{\alpha}|U|^{\frac{1}{\beta}}$.
	
	Let $W=U\cup V$. Then $W$ is a finite generating set of $G$. Suppose that $|W^n|\geqslant (\alpha|W|)^{\beta n}$ for every $n\in\mathbb{N}$. Note that $|W|\geqslant|V|$, so $|W^n|\geqslant(\alpha|V|)^{\beta n}$.
	
	We can also note that, as $V$ is closed under conjugation by elements of $G$, we have that $W^n= U^n\cup U^{n-1}V\cup\cdots \cup UV^{n-1}\cup V^n$. We can see this by considering $w_1,\ldots,w_n\in W$, and supposing that $w_k\notin U$. Then $w_k\in V$, and $w_1\cdots w_n=w_1\cdots w_{k-1}w_{k+1}\cdots w_nw'_k$, where $w'_k=(w_{k+1}\cdots w_n)^{-1}w_kw_{k+1}\cdots w_n$. As this is a conjugate of $w_k\in V$, we have that $w'_k\in V$. We can therefore say that $|W^n|\leqslant (n+1)|U^n||V^n|\leqslant (n+1)|U|^n|V^n|$.
	
	By our assumption, we therefore have that $|V^n|\geqslant\frac{(\alpha|V|)^{\beta n}}{(n+1)|U|^n}$. Hence
	\begin{align*}
		\omega(\langle V\rangle,V) & =\lim_{n\to\infty}|B_V(n)|^{\frac{1}{n}}
		\\ & \geqslant \lim_{n\to\infty}|V^n|^{\frac{1}{n}}
		\\ & \geqslant \lim_{n\to\infty}\frac{(\alpha|V|)^{\beta}}{(n+1)^{\frac{1}{n}}|U|}
		\\ & = \frac{(\alpha|V|)^{\beta}}{|U|}
		\\ & >1.
	\end{align*}
	
	Therefore $\langle V\rangle$ has exponential growth. However $\langle V\rangle \leqslant H$, so this is a contradiction. Hence $|W^n|< (\alpha|W|)^{\beta n}$ for some $n\in\mathbb{N}$.
\end{proof}

The following corollary was partially suggested by the referee.

\begin{cor}
	\label{InfFICentre}
	Let $G$ be a finitely generated group, and suppose that $G$ has a finite index subgroup $H$ with infinite centre. Then for any constants $\alpha,\beta>0$, we can find a finite generating set $U$ of $G$ such that $|U^n|< (\alpha|U|)^{\beta n}$ for some $n\in\mathbb{N}$.
\end{cor}

\begin{proof}
	It is a standard fact that $H$ contains a subgroup $N$, such that $N$ is normal and finite index in both $H$ and $G$. Let $Z(H)$ denote the centre of $H$, and $Z(N)$ denote the centre of $N$. As $Z(H)\cap N\leqslant Z(N)$, $Z(H)\cap N$ has finite index in $Z(H)$, and $Z(H)$ is infinite, we have that $Z(N)$ is infinite.
	
	We now note that $Z(N)$ is an abelian group, so finitely generated subgroups of $Z(N)$ will not have exponential growth. As $Z(N)$ is characteristic in $N$, we get that $Z(N)$ is a normal subgroup of $G$.
	
	Note that for $h\in Z(N)$, the conjugacy class of $h$ in $G$ is the orbit of $h$ under the conjugation action, and $N$ is a subgroup of the stabiliser of $h$. The Orbit-Stabiliser theorem therefore tells us that the size of any such conjugacy class is bounded by the index of $N$ in $G$. We can now apply \Cref{InfNormal} to get the result.
\end{proof}

An important example of where the conditions in \Cref{InfFICentre} are satisfied is when we have a direct product with a free abelian group.

\begin{lem}
	\label{ZInjective}
	Let $G$ be a finitely generated subgroup of $H\times \mathbb{Z}^m$ for some $m\in\mathbb{N}$, where the projective map from $G$ to $H$ is not injective. Then the centre of $G$ is infinite.
\end{lem}

\begin{proof}
	As the projective map is not injective, we can find $g\in H$ and $a,b\in \mathbb{Z}^m$ such that $a\neq b$ and $(g,a),(g,b)\in G$. Hence $(e,a^{-1}b)\in G$, and so $\{(e,(a^{-1}b)^k):k\in\mathbb{Z}\}$ is an infinite central subgroup of $G$.
\end{proof}

\subsubsection{Direct products with non-exponential factors}

\Cref{ZInjective} shows that certain subgroups of direct products with $\mathbb{Z}^m$ do not have uniform product set growth. This is clearly not necessarily true for all subgroups of direct products.

\begin{exm}
	\label{InjectiveExample}
	Consider $F_2\times \mathbb{Z}$, where $F_2=\langle a,b\rangle$. Let $H=\langle (a,1),(b,1)\rangle$. Then $H$ is a subgroup of $F_2\times\mathbb{Z}$ with non-trivial projection to $\mathbb{Z}$, however for each $g\in F_2$ there exists a unique $n\in\mathbb{Z}$ such that $(g,n)\in H$. In other words, $H\cong F_2$, and so $H$ has uniform product set growth by the main theorem of \cite{Safin2011}.
\end{exm}

We would like to understand more about when direct products and their subgroups have uniform product set growth, and when they do not. We begin by considering a direct product where one of the factors does not have uniform exponential growth.

\begin{prop}
	\label{NoProducts}
	Let $H_1\times H_2$ be a finitely generated direct product of groups such that $H_2$ is infinite and does not have uniform exponential growth. Then for any constants $\alpha,\beta>0$ we can find a finite generating set $U$ of $H_1\times H_2$ such that $|U^n|< (\alpha|U|)^{\beta n}$ for some $n\in\mathbb{N}$.
\end{prop}

\begin{proof}
	Fix constants $\alpha,\beta>0$. Suppose that for every finite generating set $U$ of $H_1\times H_2$, we have that $|U^n|\geqslant (\alpha|U|)^{\beta n}$ for every $n\in\mathbb{N}$. We will show that this gives a contradiction to $H_2$ not having uniform exponential growth.
	
	We fix a finite $V_1\subset H_1$ such that $\langle V_1\rangle=H_1$, and then let $V_2\subset H_2$ be finite with $\langle V_2\rangle=H_2$. Let $B=B_{V_1}(1)\times B_{V_2}(1)$, then we have that $\langle B\rangle=H_1\times H_2$. We note that $B^n=B_{V_1}(n)\times B_{V_2}(n)$ for every $n\in\mathbb{N}$, so $|B^n|=|B_{V_1}(n)||B_{V_2}(n)|$.
	
	As $B$ generates $H_1\times H_2$, by our assumption $|B^n|\geqslant (\alpha|B|)^{\beta n}$ for every $n\in\mathbb{N}$, and every possible choice of $V_2$. This means that for every $n\in\mathbb{N}$ we have that
	\begin{align*}
	|B^n|\geqslant (\alpha|B|)^{\beta n} & \implies |B_{V_1}(n)||B_{V_2}(n)|\geqslant (\alpha|B_{V_2}(1)|)^{\beta n}
	\\ & \implies \lim_{n\to\infty}(|B_{V_1}(n)||B_{V_2}(n)|)^{\frac{1}{n}}\geqslant (\alpha|B_{V_2}(1)|)^{\beta}
	\\ & \implies \omega(H_1,V_1)\omega(H_2,V_2)\geqslant (\alpha|B_{V_2}(1)|)^{\beta}
	\\ & \implies \omega(H_2,V_2)\geqslant \frac{(\alpha|B_{V_2}(1)|)^{\beta}}{\omega(H_1,V_1)}.
	\end{align*}
	In particular, this means that if $|B_{V_2}(1)|\geqslant \frac{(2\omega(H_1,V_1))^{\frac{1}{\beta}}}{\alpha}$, then $\omega(H_2,V_2)\geqslant 2$. We let $k=\lceil\frac{(2\omega(H_1,V_1))^{\frac{1}{\beta}}}{\alpha}\rceil$, recalling that $V_1$ was fixed.
	
	For an arbitrary finite $V_2\subset H_2$ such that $\langle V_2\rangle=H_2$ we let $\overline{V_2}=B_{V_2}(k)$. Note that this is also a finite generating set of $H_2$, with the additional property that $|B_{\overline{V_2}}(1)|=|B_{V_2}(k)|\geqslant k$. We therefore have that $\omega(H_2,\overline{V_2})\geqslant 2$, and so
	\begin{equation*}
	2\leqslant\omega(H_2,\overline{V_2})=\lim_{n\to\infty}|B_{\overline{V_2}}(n)|^{\frac{1}{n}}=\lim_{n\to\infty}|B_{V_2}(nk)|^{\frac{1}{n}}=\omega(H_2,V_2)^k,
	\end{equation*}
	so $\omega(H_2,V_2)\geqslant 2^{\frac{1}{k}}$. Recalling that $k$ was fixed, we therefore have that $\omega(H_2)\geqslant 2^{\frac{1}{k}}>1$, so $H_2$ has uniform exponential growth. This is a contradiction, so there must exist some finite generating set $U$ of $H_1\times H_2$ such that $|U^n|< (\alpha|U|)^{\beta n}$ for some $n\in\mathbb{N}$
\end{proof}

The key to the above proof was being able to say something about the growth of every ball in $H_2$. This will usually not be possible using a subgroup of such a direct product, as it may not be able to give us sufficient information about $H_2$ to determine whether it has uniform exponential growth or not. On the other hand, if we instead consider a direct product where one of the factors does not have exponential growth, a similar idea can still work.

For a direct product $H_1\times H_2$, consider the projection map $\varphi:H_1\times H_2\to H_1$. For $U\subset H_1\times H_2$, we can also consider the restriction $\varphi|_{\langle U\rangle}$. This restricted map is a group homomorphism. In \Cref{InjectiveExample}, this homomorphism is injective, allowing us to see $\langle U\rangle$ as a subgroup of $H_1$, and effectively allowing us to ignore the direct product.

More generally, the fact that this map is a homomorphism means that it will always be $k$-to-1, where $k$ is the cardinality of $\ker(\varphi|_{\langle U\rangle})$. So long as this cardinality is finite, it turns out that uniform product set growth for $\langle U\rangle$ can be inherited from its projection to $H_1$.

\begin{lem}
	\label{BoundedToOne}
	Let $H_1$ and $H_2$ be groups, and let $U\subset H_1\times H_2$ be finite. Suppose that for some $\alpha,\beta>0$ the projection $U_1$ of $U$ to $H_1$ satisfies $|U_1^n|\geqslant (\alpha|U_1|)^{\beta n}$ for every $n\in\mathbb{N}$, and that the projective map from $\langle U\rangle$ to $H_1$ is $k$-to-1 for some $k\in\mathbb{N}$. Then $|U^n|\geqslant (\frac{\alpha}{k}|U|)^{\beta n}$ for every $n\in\mathbb{N}$.
\end{lem}

\begin{proof}
	As the projective map is $k$-to-1 we have that $|U|= k|U_1|$, so combined with the fact that $|U^n|\geqslant |U^n_1|$, we obtain our inequality.
\end{proof}

\begin{rem}
	Note that if one of the factors in a direct product is finite, then the projection to the other factors will always be finite-to-1. Consequently, in this case a subgroup will have uniform product set growth if and only if its projection to the other factors has uniform product set growth \cite[Lemma 3.6.7]{Kerr2022}. This is why the assumption of $H_2$ being infinite was necessary in \Cref{NoProducts}.
\end{rem}

On the other hand, if the cardinality of the projection is infinite, then this can form an obstruction to uniform product set growth.

\begin{prop}
	\label{InfToOne}
	Let $H_1$ and $H_2$ be groups, where no finitely generated subgroup of $H_2$ has exponential growth. Let $G\leqslant H_1\times H_2$ be finitely generated, such that the projective map from $G$ to $H_1$ is infinite-to-1. Then for any constants $\alpha,\beta>0$ we can find a finite generating set $U$ of $G$ such that $|U^n|< (\alpha|U|)^{\beta n}$ for some $n\in\mathbb{N}$.
\end{prop}

\begin{proof}
	Fix constants $\alpha,\beta>0$. Suppose that for every finite generating set $U$ of $G$, we have that $|U^n|\geqslant (\alpha|U|)^{\beta n}$ for every $n\in\mathbb{N}$. We will show that this gives a contradiction to $H_2$ not having exponential growth.
	
	Let $G_{H_1}$ and $G_{H_2}$ be the projections of $G$ to $H_1$ and $H_2$ respectively. Note that as the projection of $G$ to $G_{H_1}$ is infinite-to-1, for every $g\in G_{H_1}$ there exists infinitely many $h\in H_2$ such that $(g,h)\in G$.
	
	We fix a finite $V\subset G$ such that $\langle V\rangle=G$. Let $V_{H_1}$ and $V_{H_2}$ be its projections to $H_1$ and $H_2$ respectively. As $V_{H_2}$ is finite we can write it as $V_{H_2}=\{v_1,\ldots,v_m\}$. Pick $g\in V_{H_1}$, then for some $k\in\mathbb{N}$ choose $\{\overline{v}_1,\ldots,\overline{v}_k\}\subset H_2$ such that $\{v_1,\ldots,v_m\}\cap\{\overline{v}_1,\ldots,\overline{v}_k\}=\emptyset$ and $(g,\overline{v}_i)\in G$ for every $i\in\{1,\ldots,k\}$. Let $\overline{V}=V\cup\{(g,\overline{v}_i):i\in\{1,\ldots,k\}\}$. Then $\langle\overline{V}\rangle=\langle V\rangle=G$, where $\overline{V}_{H_1}=V_{H_1}$ and $|\overline{V}_{H_2}|\geqslant k$. Note that the choice of $k$ here was arbitrary.
	
	As $\overline{V}$ generates $G$, by our assumption $|\overline{V}^n|\geqslant (\alpha|\overline{V}|)^{\beta n}$ for every $n\in\mathbb{N}$. Note that
	\begin{equation*}
	|\overline{V}_{H_2}^n|\leqslant |\overline{V}^n|\leqslant |\overline{V}_{H_1}^n||\overline{V}_{H_2}^n|=|V_{H_1}^n||\overline{V}_{H_2}^n|\leqslant|B_{V_{H_1}}(n)||B_{\overline{V}_{H_2}}(n)|.
	\end{equation*}
	Hence, by the same reasoning as in \Cref{NoProducts}, we get that
	\begin{equation*}
		\omega(G_{H_2},\overline{V}_{H_2})	\geqslant\frac{(\alpha|\overline{V}_{H_2}|)^{\beta}}{\omega(G_{H_1},V_{H_1})}.
	\end{equation*}
	In particular, if we pick $k$ such that $k\geqslant \frac{(2\omega(G_{H_1},V_{H_1}))^{\frac{1}{\beta}}}{\alpha}$, then we have that $|\overline{V}_{H_2}|\geqslant \frac{(2\omega(G_{H_1},V_{H_1}))^{\frac{1}{\beta}}}{\alpha}$, which implies that $\omega(G_{H_2},\overline{V}_{H_2})\geqslant 2$. This tells us that $G_{H_2}$ has exponential growth, which is a contradiction, as $G_{H_2}$ is a finitely generated subgroup of $H_2$. We conclude that there must exist some finite generating set $U$ of $G$ such that $|U^n|< (\alpha|U|)^{\beta n}$ for some $n\in\mathbb{N}$.
\end{proof}


\subsubsection{Direct products of groups with uniform product set growth}

In the previous subsection we dealt with the case where we have a subgroup of a direct product where one of the factor groups does not have exponential growth. We will now consider a very different direct product case, which is a subgroup of a direct product where all of the factors have uniform product set growth. We will show that the growth of the factors passes to the growth of the subgroup, so long as the number of factors is bounded. The first thing to note here is the standard fact that if we have a subset of a direct product, we can always relate the size of the original set to the size of one of its projections.

\begin{lem}
	\label{LargeProjection}
	Let $G=G_1\times \cdots \times G_m$ be a group, and let $U\subset G$ be finite, with $U_i$ the projection of $U$ to the factor $G_i$. Then $\max\{|U_1|,\ldots,|U_m|\}\geqslant |U|^{\frac{1}{m}}$.
\end{lem}

\begin{proof}
	This follows from noting that $U\subset U_1\times\cdots \times U_m$.
\end{proof}

This almost immediately gives us that if all of the projections of a subset $U$ of a direct product satisfy the uniform product set growth inequality, then so will $U$, with a change of constants that depends on the number of factors in the direct product.

\begin{cor}
	\label{Factors1}
	Let $G=G_1\times \cdots \times G_m$ be a group, and let $U\subset G$ be finite, with $U_i$ the projection of $U$ to the factor $G_i$. Suppose there exist $\alpha,\beta>0$ such that $|U_i^n|\geqslant(\alpha|U_i|)^{\beta n}$ for every $n\in\mathbb{N}$ and $i\in\{1,\ldots,m\}$. Then $|U^n|\geqslant(\alpha^m|U|)^{\frac{\beta n}{m}}$ for every $n\in\mathbb{N}$.
\end{cor}

\begin{proof}
	By \Cref{LargeProjection} we choose $U_i$ such that $|U_i|\geqslant |U|^{\frac{1}{m}}$. We can then conclude that $|U^n|\geqslant |U_i^n|\geqslant (\alpha|U_i|)^{\beta n}\geqslant (\alpha|U|^{\frac{1}{m}})^{\beta n}=(\alpha^m|U|)^{\frac{\beta n}{m}}$ for every $n\in\mathbb{N}$.
\end{proof}

The above result tells us that if we know about the growth of the projections of a specific subset, then we can say something about the growth of the whole subset. This naturally means that if we can say something general about the growth of subsets of the factor groups, then we can also say something general about the growth of subsets of the direct product. Using \cite[Theorem 1.1]{Delzant2020}, one place we can apply this idea is to subgroups of direct products of hyperbolic groups.

\begin{defn}
	We say a group is \emph{contained non-trivially in a direct product} when it is a subgroup of a direct product, and none of its projections to any of the factor groups are trivial.
\end{defn}

\begin{prop}
	\label{HypFactors}
	Let $G=G_1\times \cdots \times G_m$ be a direct product of hyperbolic groups. There exists a constant $\alpha>0$ such that for every finite $U\subset G$, at least one of the following must hold:
	\begin{enumerate}
		\item $\langle U\rangle$ is virtually $\mathbb{Z}$.
		\item $\langle U\rangle$ is contained non-trivially in a direct product of the form $H_1\times H_2$, where $H_2$ is virtually $\mathbb{Z}$.
		\item $|U^n|\geqslant(\alpha|U|)^{\frac{n}{2m}}$ for every $n\in\mathbb{N}$.
	\end{enumerate}
\end{prop}

\begin{proof}
	The statement of \cite[Theorem 1.1]{Delzant2020} says that for each hyperbolic group $G_i$ there exists $\alpha_i>0$ such that for every finite $V_i\subset G_i$ where $\langle V_i\rangle$ is not virtually $\mathbb{Z}$, we have that $|V_i^n|\geqslant(\alpha_i|V_i|)^{\frac{n}{2}}$ for every $n\in\mathbb{N}$. Let $\alpha=\min\{\alpha_1,\ldots,\alpha_m\}$.
	
	Suppose $U\subset G$ is finite, where $\langle U\rangle$ is not virtually $\mathbb{Z}$, and is not contained non-trivially in a direct product $H_1\times H_2$ where $H_2$ is virtually $\mathbb{Z}$. Then every projection $U_i$ to $G_i$ has the property that $|U_i^n|\geqslant(\alpha|U_i|)^{\frac{n}{2}}$ for every $n\in\mathbb{N}$. Therefore $|U^n|\geqslant(\alpha^m|U|)^{\frac{ n}{2m}}$ for every $n\in\mathbb{N}$, by \Cref{Factors1}.
\end{proof}

We remark here that \Cref{HypFactors} does not give us a dichotomy of subgroups, as demonstrated by \Cref{InjectiveExample}. We might be tempted to try to use the work in the previous subsection to get such a dichotomy, however in the general case this is not possible, as using finite-to-1 projections to get down to the factors with uniform product set growth will only work if we have some control over the size of the projections. On the other hand, it turns out we do have some control over these projections in the case of virtually torsion-free groups.

\begin{lem}
	\label{InfOrBounded}
	Let $H_1$ and $H_2$ be groups, where $H_2$ has a torsion-free subgroup $H_2'$ of index $d\in\mathbb{N}$. Let $G\subset H_1\times H_2$ be finitely generated, such that the projective map from $G$ to $H_1$ is $k$-to-1. Then either $k\leqslant d$, or $k$ is infinite.
\end{lem}

\begin{proof}
	Let $G_{H_1}$ and $G_{H_2}$ be the projections of $G$ to $H_1$ and $H_2$ respectively. Suppose that the projective map is $k$-to-1 for some $k$ such that $k>d$, or is infinite. We will show that it must be infinite.
	
	The kernel of the projective map has cardinality $k$, so by the pigeonhole principle there must exist $a,b\in G_{H_2}$ such that $a\neq b$, $(e,a),(e,b)\in G$, and $a,b\in hH_2'\subset H_2$. That is, they are in the same coset of $H_2'$ in $H_2$.
	
	This means that there exist $n_a,n_b\in H_2'$ such that $n_a\neq n_b$ and $a=hn_a$, $b=hn_b$, so $a^{-1}b=n_a^{-1}n_b\neq e$. Therefore $(e,n_a^{-1}n_b)\in G$, and has infinite order as $H_2'$ is torsion-free. In particular this is in the kernel of the projective map, so the kernel must have infinite cardinality.
\end{proof}

We can therefore get a dichotomy for the subgroups of direct products of virtually torsion-free hyperbolic groups. The fact that the statement below gives us a dichotomy follows from \Cref{InfToOne}.

\begin{prop}
	Let $G=G_1\times \cdots \times G_m$ be a direct product of virtually torsion-free hyperbolic groups. There exist $\alpha,\beta>0$ such that for every finite $U\subset G$, at least one of the following must hold:
	\begin{enumerate}
		\item $\langle U\rangle$ is contained a direct product of the form $H_1\times H_2$, where $H_2$ is virtually free abelian and the projection of $\langle U\rangle$ to $H_1$ is infinite-to-1.
		\item $|U^n|\geqslant(\alpha|U|)^{\frac{n}{2m}}$ for every $n\in\mathbb{N}$.
	\end{enumerate}
\end{prop}

\begin{proof}
	For every $G_i$, there exists a torsion-free subgroup $G_i'$ of finite index. Let that index be $d_i\in \mathbb{N}$. Then $G_1'\times \cdots\times G_m'$ is a torsion-free subgroup of $G$ with index $d=d_1\ldots d_m$. Hence every subgroup of $G$ has a torsion-free subgroup with index at most $d$.
	
	Let $U\subset G$ be finite. Assume that if $\langle U\rangle$ is contained a direct product of the form $H_1\times H_2$, where $H_2$ is virtually $\mathbb{Z}$, then the projection of $\langle U\rangle$ to $H_1$ is finite-to-1. We note that this allows us to rule out the case that $\langle U\rangle$ is virtually free abelian, as otherwise we could see $\langle U \rangle$ as a subgroup of $\{e\}\times\langle U\rangle$, where the projection to $\{e\}$ is infinite-to-1.
	
	Let $U_i$ be the projection of $U$ to $G_i$, so $\langle U\rangle\leqslant \langle U_1\rangle\times\cdots\times\langle U_m\rangle$. We rewrite the $U_i$ as $V_1,\ldots,V_k,W_1,\ldots,W_{m-k}$, where the $\langle V_i\rangle$ are not virtually $\mathbb{Z}$, and the $\langle W_i\rangle$ are virtually $\mathbb{Z}$. Note that the set $\{V_1,\ldots,V_k\}$ is nonempty, as otherwise $\langle U\rangle$ would be an infinite subgroup of a virtually free abelian group, and so would itself be virtually free abelian. On the other hand, $\{W_1,\ldots,W_{m-k}\}$ may be empty.
	
	Let $U'$ be the projection of $U$ to $\langle V_1\rangle\times\cdots\times\langle V_k\rangle$. The projection of $U'$ to each of these factors is simply $V_i$. As $\langle V_i\rangle$ is not virtually $\mathbb{Z}$, \cite[Theorem 1.1]{Delzant2020} tells us there exists $\alpha_i>0$, dependent only on $G_i$, such that $|V_i^n|\geqslant(\alpha_i|V_i|)^{\frac{n}{2}}$ for every $n\in\mathbb{N}$. Let $\alpha=\min\{\alpha_1,\ldots,\alpha_m\}$. Noting that $k\leqslant m$, we have that $|(U')^n|\geqslant (\alpha^m|U'|)^{\frac{n}{2m}}$ for every $n\in\mathbb{N}$ by \Cref{Factors1}.
	
	We know that $\langle W_1\rangle\times\cdots\times\langle W_{m-k}\rangle$ has a torsion-free subgroup of index at most $d$, so as the projection of $\langle U\rangle$ to $\langle V_1\rangle\times\cdots\times\langle V_k\rangle$ is finite-to-1, by \Cref{InfOrBounded} it is at most $d$-to-1. It therefore follows that $|U^n|\geqslant (\frac{\alpha^m}{d}|U|)^{\frac{n}{2m}}$ for every $n\in\mathbb{N}$ by \Cref{BoundedToOne}.
\end{proof}

\begin{rem}
	Whether all hyperbolic groups are virtually torsion-free or not is an open question, and is equivalent to the problem of whether all hyperbolic groups are residually finite \cite{Kapovich2000}.
\end{rem}

The method used in the above proof is similar to the methods we will use in Section 4, where we will be considering right-angled Artin groups and mapping class groups, which are virtually torsion-free and have bounds of the number of factors in any subgroup that is a direct product.


\subsubsection{Passing to finite index supergroups}

We now move away from direct products, and consider another way that uniform product set growth can be inherited. It is well known that uniform exponential growth passes to finite index supergroups \cite[Proposition 3.3]{Shalen1992}. This can be shown by checking that for any finite generating set of the supergroup, there exists a finite generating set of our finite index subgroup where each generator has bounded length. We can adapt this method to show that uniform product set growth also passes to finite index supergroups, at least in the case of symmetric sets.

We note here that the question of whether uniform exponential growth passes to finite index subgroups is still open, and this is also open for uniform product set growth. A partial answer is given by \Cref{InfFICentre}.

To prove that uniform product set growth passes to finite index supergroups, we first need to collect some preliminary results.

\begin{lem}
	\emph{\cite[Proposition 2.3]{Mann2012}}
	\label{CosetRep}
	Let $G$ be a finitely generated group, and let $U$ be a finite generating set of $G$. Let $H\leqslant G$ such that $[G:H]=d$. There exists a set of representatives of the right cosets of $H$ such that each representative has length at most $d-1$ in the word metric induced by $U$.
\end{lem}

The following is not stated explicitly in \cite{Mann2012}, but follows from the work there.

\begin{prop}
	\label{FGSet}
	Let $G$ be a finitely generated group, and let $U$ be a finite generating set of $G$. Let $H\leqslant G$ such that $[G:H]=d$. There exists a finite generating set $V$ of $H$ such that $|V|\geqslant\frac{1}{d}|U|$ and $V\subset B_U(2d-1)$.
\end{prop}

\begin{proof}
	Let $\{g_1,\ldots,g_d\}\subset G$ be a set of representatives of the right cosets of $H$ such that each representative has length at most $d-1$ in the word metric induced by $U$, which we know exists by \Cref{CosetRep}. We can assume that $1\in\{g_1,\ldots,g_d\}$.
	
	For each $i\in\{1,\ldots,d\}$ and $u\in U\cup U^{-1}$, let $\sigma(i,u)\in \{1,\ldots,d\}$ be such that $g_{\sigma(i,u)}$ is the right coset representative of $Hg_iu$. Proposition 2.1 in \cite{Mann2012} tells us that the set of elements $V=\{g_iug_{\sigma(i,u)}^{-1}:i\in\{1,\ldots,d\},u\in U\cup U^{-1}\}$ is a finite generating set for $H$. Note that for $u,v\in U$ such that $u\neq v$, we have that $g_1ug_{\sigma(1,u)}^{-1}=g_1vg_{\sigma(1,v)}^{-1}$ implies that $\sigma(1,u)\neq \sigma(1,v)$. This means that for any $d+1$ elements in $U$, we can construct at least two different elements in $V$, and so forth. Therefore $|V|\geqslant \frac{1}{d}|U|$. We also note that each $g_iug_{\sigma(i,u)}^{-1}$ has length at most $2d-1$ in the word metric induced by $U$, so $V\subset B_U(2d-1)$.
\end{proof}

\begin{lem}
	\label{SymGrowth}
	Let $U$ be a finite symmetric subset of a group $G$. Then $|U^n|\leqslant |B_U(n)|\leqslant 2|U^n|$.
\end{lem}

\begin{proof}
	The lower bound follows from the fact that $U^n\subset B_U(n)$. For the upper bound, $U$ being symmetric means that $U^n\subset U^{n+2}$ for every $n\in\mathbb{N}$. Therefore
	\begin{equation*}
		B_U(n)=\{e\}\cup U\cup\cdots\cup U^n=U^{n-1}\cup U^n
	\end{equation*}
	for any $n\in\mathbb{N}$. This means that $|B_U(n)|\leqslant |U^{n-1}|+|U^n|$, so as $|U^{n-1}|\leqslant |U^n|$ we get that $|B_U(n)|\leqslant 2|U^n|$.
\end{proof}

We can now prove that uniform product set growth passes to finite index supergroups.

\begin{prop}
	\label{Supergroup}
	Let $G$ be a finitely generated group, and let $U$ be a finite symmetric generating set of $G$. Let $H\leqslant G$ such that $[G:H]=d$. Suppose that there exist $\alpha,\beta>0$ such that for any finite symmetric generating set $V$ of $H$, we have that $|V^n|\geqslant(\alpha|V|)^{\beta n}$ for every $n\in\mathbb{N}$. Let $m=2d-1$. Then 
	\begin{equation*}
		|U^n|\geqslant \bigg(\frac{\alpha}{2^{\frac{m}{\beta}}d}|U|\bigg)^{\frac{\beta n}{m}}
	\end{equation*}
	for every $n\in\mathbb{N}$.
\end{prop}

\begin{proof}
	We know from \Cref{FGSet} that there exists a finite symmetric generating set $V$ of $H$ such that $|V|\geqslant\frac{1}{d}|U|$ and $V\subset B_U(m)$. Let $n\in\mathbb{N}$. Using \Cref{SymGrowth} we know that $|B_U(n)|\leqslant 2|U^n|$, as $U$ is symmetric.
	
	For every $n\in\mathbb{N}$ we have that
	\begin{align*}
		(2|U^n|)^m & \geqslant |B_U(n)|^m\geqslant|B_U(n)^m|=|B_U(m)^n|
		\\ & \geqslant |V^n|\geqslant (\alpha|V|)^{\beta n}\geqslant\bigg(\frac{\alpha}{d}|U|\bigg)^{\beta n}.
	\end{align*}
	Therefore
	\begin{equation*}
		|U^n|\geqslant \frac{1}{2}\bigg(\frac{\alpha}{d}|U|\bigg)^{\frac{\beta n}{m}}=\bigg(\frac{\alpha}{2^{\frac{m}{\beta n}}d}|U|\bigg)^{\frac{\beta n}{m}}
	\end{equation*}
	It now only remains to note that $2^{\frac{1}{n}}\leqslant 2$ for every $n\in\mathbb{N}$, and the conclusion follows.
\end{proof}

The above result has recently been generalised to the non-symmetric case \cite[Lemma 5.6]{Wan2023}.

\section{Growth in acylindrically hyperbolic groups}

The focus of this section will be to prove some results about product set growth for certain subsets of acylindrically hyperbolic groups. This will provide us with some of the tools that we will need in Section 4 when trying to answer \Cref{MainQuestion} for virtual subgroups of mapping class groups.

In Section 3.1, we will give some necessary background information regarding quasi-trees, hyperbolic spaces, and acylindrically hyperbolic groups. In Section 3.2, we will use the fact that tree approximation is uniform in quasi-trees \cite{Kerr2023} to get a generalisation of Delzant and Steenbock's result for groups acting acylindrically on trees \cite[Theorem 1.11]{Delzant2020}, to groups acting acylindrically on quasi-trees. As noted below in \Cref{Balasubramanya}, it was shown by Balasubramanya that every acylindrically hyperbolic group has an acylindrical action on a quasi-tree. Hence this generalisation to quasi-trees can be seen as being complementary to Delzant and Steenbock's statement for groups acting acylindrically on hyperbolic spaces \cite[Theorem 1.14]{Delzant2020}.

We will also use Section 3.2 to give a more concrete method by which we can use this generalisation to answer \Cref{MainQuestion} for certain subgroups of acylindrically hyperbolic groups, namely by quickly finding loxodromic elements in the action on a quasi-tree. In Section 3.3 we will note that, by a result of Fujiwara \cite{Fujiwara2021}, this method can also be used when the acylindrical action is on a hyperbolic space. This is a more practical setting for many groups, although it comes at the cost of having to restrict to symmetric subsets.

\subsection{Acylindrically hyperbolic groups}

\subsubsection{Hyperbolic spaces and quasi-trees}

\begin{defn}
	Let $(X,d)$ be a metric space. Let $x_0,x,y\in X$. The \emph{Gromov product} of $x$ and $y$ at $x_0$ is
	\begin{equation*}
	(x,y)_{x_0}=\frac{1}{2}(d(x_0,x)+d(x_0,y)-d(x,y)).
	\end{equation*}
\end{defn}

\begin{note}
	For $x,y$ in a geodesic metric space, we will use the notation $[x,y]$ to represent any geodesic from $x$ to $y$.
\end{note}

\begin{defn}
	Let $(X,d)$ be a geodesic metric space. We say that $X$ is \emph{$\delta$-hyperbolic}, for some $\delta\geqslant 0$, if for every $x,y,z\in X$ any choice of geodesic triangle $[x,y]\cup[y,z]\cup[z,x]$ is \emph{$\delta$-slim}, meaning that if $p\in[x,y]$ then there exists $q\in[y,z]\cup[z,x]$ such that $d(p,q)\leqslant \delta$.
\end{defn}

\begin{con}
	Every hyperbolic space we work with in this paper will be assumed to be geodesic.
\end{con}


\begin{rem}
	A 0-hyperbolic space is an $\mathbb{R}$-tree \cite[pg.~42]{Coornaert1990}.
\end{rem}

A quasi-tree is most commonly defined as being a geodesic metric space that is quasi-isometric to a simplicial tree, which is equivalent to being quasi-isometric to an $\mathbb{R}$-tree \cite{Kerr2023}. Here, however, we will use Manning's characterisation of these spaces.

\begin{thm}[Manning's bottleneck criterion]
	\emph{\cite{Manning2005}}
	\label{Bottleneck}
	A geodesic metric space $(X,d)$ is a quasi-tree if and only if there exists a constant $\Delta\geqslant 0$ (the bottleneck constant) such that for every geodesic $[x,y]$ in $X$, and every $z\in[x,y]$, every path between $x$ and $y$ intersects the closed ball $B(z,\Delta)$.
\end{thm}

\begin{rem}
	\label{DeltaHyperbolic}
	It follows from \Cref{Bottleneck} that a quasi-tree with bottleneck constant $\Delta\geqslant 0$ is $\Delta$-hyperbolic.
\end{rem}

A key property of hyperbolic spaces is Gromov's tree approximation lemma, which essentially states that every finite subset of a hyperbolic space can be approximated by an $\mathbb{R}$-tree, with error dependent on the number of points being approximated. In quasi-trees this approximation is uniform, in the sense that the error does not depend on the subset being approximated.

\begin{prop}
	\label{Uniform}
	\emph{\cite{Kerr2023}}
	Let $(X,d)$ be a $\delta$-hyperbolic quasi-tree, with bottleneck constant $\Delta\geqslant 0$. Let $x_0\in X$, and let $Z\subset X$. Let $Y$ be a union of geodesic segments $\bigcup_{z\in Z}[x_0,z]$. Then there is an $\mathbb{R}$-tree $T$ and a map $f:(Y,d)\to (T,d^*)$ such that:
	\begin{enumerate}
		\item For all $z\in Z$, the restriction of $f$ to the geodesic segment $[x_0,z]$ is an isometry.
		\item For all $x,y\in Y$, we have that $d(x,y)-2(\Delta+2\delta)\leqslant d^*(f(x),f(y))\leqslant d(x,y)$.
	\end{enumerate}
\end{prop}

\subsubsection{Acylindrical actions and loxodromic elements}

When we look at groups acting on quasi-trees and hyperbolic spaces, we will be specifically interested in the cases where that action is acylindrical.
The following terminology comes from \cite{Dahmani2017}.

\begin{defn}
	\label{AHDef2}
	We say that an isometric action of a group $G$ on a $\delta$-hyperbolic space is \emph{acylindrical}, or \emph{$(\kappa_0,N_0)$-acylindrical} for $\kappa_0\geqslant\delta$ and $N_0\geqslant 1$, if for every $x,y\in X$ with $d(x,y)\geqslant \kappa_0$ we have that
	\begin{equation*}
	|\{g\in G: d(x,gx)\leqslant 100\delta\text{ and } d(y,gy)\leqslant 100\delta\}|\leqslant N_0.
	\end{equation*}
\end{defn}

\begin{rem}
	By a result in \cite{Dahmani2017}, if $\delta>0$, then this is equivalent to the more common definition of acylindricity, which can be found in papers such as \cite{Osin2016}.
\end{rem}

It turns out that every group admits an acylindrical action on a hyperbolic space, simply by considering the trivial action on a point. To have a meaningful class of groups to consider, we therefore need another condition on the action.

\begin{defn}
	An acylindrical action by a group $G$ on a hyperbolic space $X$ is \emph{non-elementary} if $G$ is not virtually cyclic and all orbits are unbounded. A group is \emph{acylindrically hyperbolic} if it admits a non-elementary acylindrical action on a hyperbolic space.
\end{defn}

This is a wide class of groups, which includes (non-virtually cyclic) hyperbolic groups, most mapping class groups, most right-angled Artin groups, and Out($F_n$) for $n\geqslant 2$. See the appendix of \cite{Osin2016} for a more complete list.

In Section 3.2, we will be interested in the acylindrically hyperbolic groups that have an acylindrical action on a quasi-tree. By a theorem of Balasubramanya, this is in fact the entire class of acylindrically hyperbolic groups. This means that any result for groups acting acylindrically on quasi-trees can automatically be applied to every acylindrically hyperbolic group, which helps motivate Section 3.2.

\begin{thm}
	\emph{\cite{Balasubramanya2017}}
	\label{Balasubramanya}
	Every acylindrically hyperbolic group admits a non-elementary acylindrical action on a quasi-tree.
\end{thm}

In Sections 3.2 and 3.3, we will show that we can say something about the growth of a subset of an acylindrically hyperbolic group if that subset contains a loxodromic element.

\begin{defn}
	Let $G$ be a group acting by isometries on a hyperbolic space $X$. We say that $g\in G$ is an \emph{elliptic} element if some (equivalently any) orbit of $g$ in $X$ has bounded diameter. We say that $g\in G$ is a \emph{loxodromic} element if it fixes exactly two points on the Gromov boundary $\partial X$. We call these fixed points $g^+$ and $g^-$.
\end{defn}

A result of Bowditch tells us that these two types classify all elements in an acylindrical action on a hyperbolic space.

\begin{prop}
	\emph{\cite{Bowditch2008}}
	Let $G$ be a group acting acylindrically on a hyperbolic space $X$. Every element $g\in G$ is either elliptic or loxodromic.
\end{prop}

%
%
%
%
%
%

We finish this section with the following notation.

\begin{defn}
	Let $G$ be a group acting acylindrically on a hyperbolic space $X$, and let $g$ be a loxodromic element in this action. Denote by $E(g)$ the maximal virtually cyclic subgroup containing $g$.
\end{defn}

\begin{rem}
	This maximal virtually cyclic subgroup exists by Lemma 6.5 in \cite{Dahmani2017}, and is exactly the stabiliser of $\{g^+,g^-\}$ in $G$.
\end{rem}

\subsection{Growth from actions on quasi-trees}

In this section we will be interested in a theorem of Delzant and Steenbock \cite[Theorem 1.14]{Delzant2020}, where they give a lower bound on the growth of $U^n$ for certain subsets of acylindrically hyperbolic groups. As a direct result of the fact that tree approximation is uniform in quasi-trees \cite{Kerr2023} we will be able to improve this theorem in the case of groups acting acylindrically on quasi-trees. By \Cref{Balasubramanya} this improvement will apply to every acylindrically hyperbolic group, as every such group admits an acylindrical action on a quasi-tree.

Before we can state the result that we wish to improve, we must first give a few definitions.

\begin{defn}
	\cite{Delzant2020}
	\label{Displacement}
	Let $U$ be a finite set of isometries of a $\delta$-hyperbolic space $(X,d)$, with $\delta>0$. The \emph{normalised $\ell^1$-energy} of $U$ is
	\begin{equation*}
	E(U)=\inf_{x\in X}\frac{1}{|U|}\sum_{u\in U}d(x,ux).
	\end{equation*}
	We then fix a base point $x_0\in X$ such that it minimises $E(U)$ up to $\delta$, so
	\begin{equation*}
	\frac{1}{|U|}\sum_{u\in U}d(x_0,ux_0)\leqslant E(U)+\delta.
	\end{equation*}
	The \emph{displacement} of $U$ is then defined to be
	\begin{equation*}
	\lambda_0(U)=\max_{u\in U}d(x_0,ux_0).
	\end{equation*}
\end{defn}

\begin{rem}
	Another version of displacement which is sometimes used is to instead take the infimum of $\max_{u\in U}d(x,ux)$ over any $x\in X$, which is clearly a lower bound for $\lambda_0(U)$. For a comparison of such quantities see \cite{Breuillard2021}.
\end{rem}

The result that Delzant and Steenbock obtained for groups acting acylindrically on hyperbolic spaces is the following.

\begin{thm}
	\label{DelzantSteenbock}
	\emph{\cite{Delzant2020}}
	Let $G$ be a group acting $(\kappa_0,N_0)$-acylindrically on a $\delta$-hyperbolic space. Assume $\kappa_0\geqslant\delta$ and $N_0\geqslant 1$. There exist constants $\alpha=\alpha(\delta,\kappa_0,N_0)>0$ and $K=K(\kappa_0)>0$ such that for every finite $U\subset G$, at least one of the following must hold:
	\begin{enumerate}
		\item $\langle U\rangle$ is virtually $\mathbb{Z}$.
		\item $\lambda_0(U)< K\log_2(2|U|)$.
		\item $|U^n|\geqslant \big(\frac{\alpha}{\log^6_2(2|U|)}|U|\big)^{\lfloor\frac{n+1}{2}\rfloor}$ for every $n\in\mathbb{N}$.
	\end{enumerate}
\end{thm}

\begin{rem}
	We note here that the action in \Cref{DelzantSteenbock} is not assumed to be non-elementary, and as stated before every group $G$ acts acylindrically on some hyperbolic space. However, if the action in question is elementary, then any finite subset $U$ of $G$ must satisfy that $\langle U\rangle$ is virtually $\mathbb{Z}$, or $\lambda_0(U)< K\log_2(2|U|)$, or be such that $|U|\leqslant \frac{1}{\alpha}$, by Proposition 4.4 in \cite{Delzant2020}.
\end{rem}

We would like to re-state this theorem for groups acting acylindrically on quasi-trees, but with the logarithm terms removed. To do this we will first check that the step in Delzant and Steenbock's proof in which the logarithm terms are introduced can be improved in the case of quasi-trees, and then combine this with the rest of their results. We will then consider the theorem we obtain in the context of \Cref{MainQuestion}, which means that we would like to understand which subsets have large enough displacement that our theorem will allow us to say something about their product set growth.

\subsubsection{The reduction lemma for quasi-trees}

The logarithm terms in \Cref{DelzantSteenbock} are introduced in a key step that Delzant and Steenbock call the reduction lemma, which we state below.

\begin{note}
	By $(Ux_0,Vx_0)_{x_0}$ $\leqslant R$ we mean that $(ux_0,vx_0)_{x_0}\leqslant R$ for every $u\in U, v\in V$, as in \cite{Delzant2020}.
\end{note}

The constant $\kappa_0$ is chosen such that $\kappa_0\geqslant\delta$, and will generally be taken from the acylindrical action on the space, as in \Cref{DelzantSteenbock}.

\begin{lem}[Reduction lemma]
	\emph{\cite{Delzant2020}}
	Let $U$ be a finite set of isometries of a $\delta$-hyperbolic metric space $X$, with $\delta> 0$. If at most $\frac{1}{4}$ of the isometries $u\in U$ have displacement $d(x_0,ux_0)\leqslant 10^{10}\kappa_0\log_2(2|U|)$, then there exist $U_0,U_1\subset U$ with cardinalities at least $\frac{1}{100}|U|$ such that
	\begin{equation*}
	(U^{-1}_0x_0,U_1x_0)_{x_0}\leqslant 1000\log_2(2|U|)\delta \text{ and } (U_0x_0,U_1^{-1}x_0)_{x_0}\leqslant 1000\log_2(2|U|)\delta.
	\end{equation*}
	In addition, for all $u_0\in U_0$ and $u_1\in U_1$ we have that
	\begin{equation*}
	d(x_0,u_0x_0)\geqslant 10^{10}\kappa_0\log_2(2|U|) \text{ and } d(x_0,u_1x_0)\geqslant 10^{10}\kappa_0\log_2(2|U|).
	\end{equation*}
\end{lem}

\begin{figure}[h]
	\centering
	\includegraphics[width=0.6\textwidth]{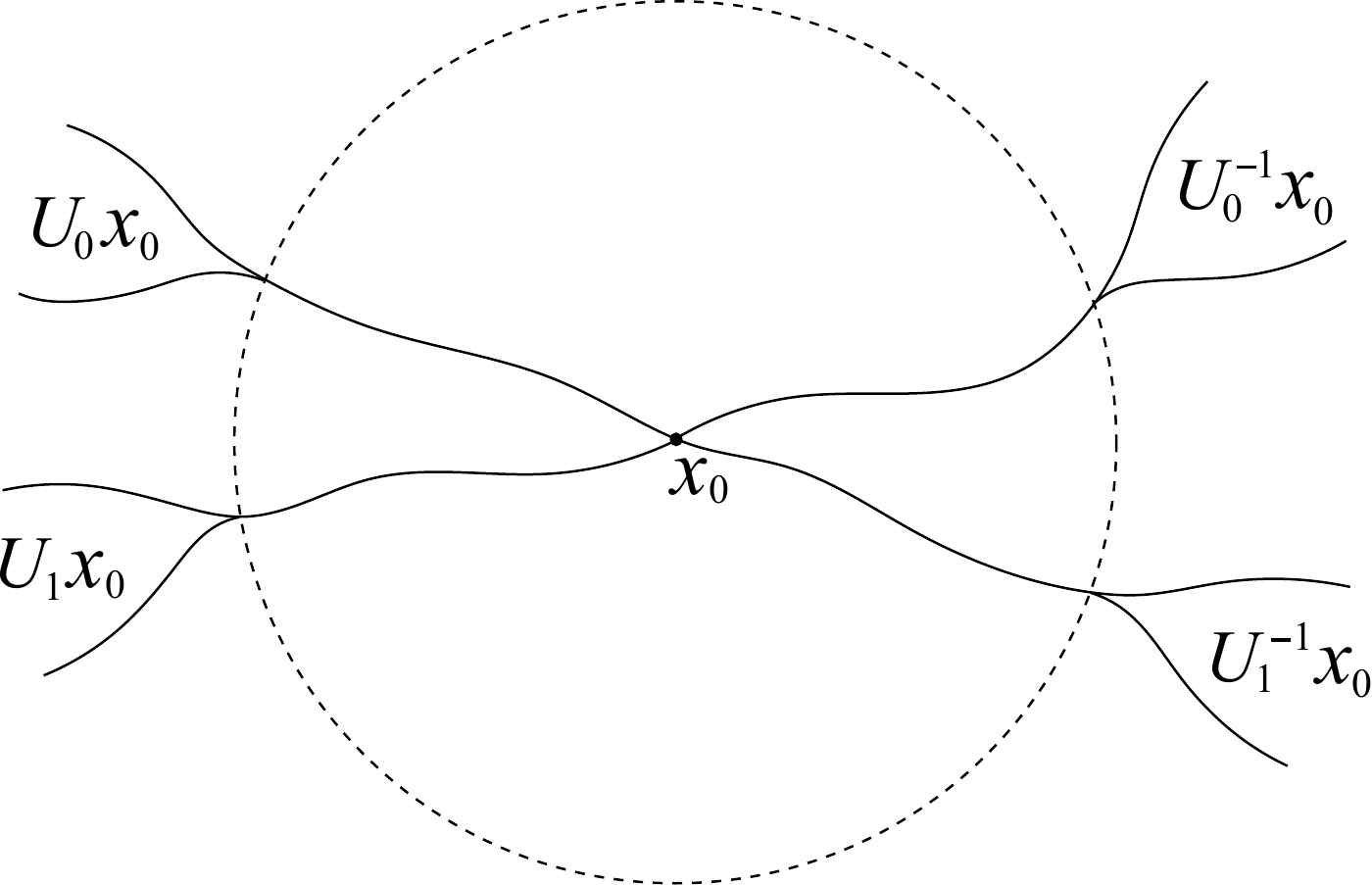}
	\caption{The reduction lemma}
	\label{ReductionFigure}
\end{figure}

\Cref{ReductionFigure} gives a representation of how such sets $U_0$ and $U_1$ may look when applied to $x_0$. Every element displaces $x_0$  by a large distance, while $U_0^{-1}$ and $U_1$ (respectively $U_0$ and $U_1^{-1}$) translate $x_0$ in opposite directions in the space. Once such sets are found, Delzant and Steenbock estimate $|U^{2n+1}|$ by instead considering $|(U_0 U_1)^nU_0|$, and showing that a lower bound on this can be given in terms of $|U_0|^n$. As $|U_0|\geqslant\frac{1}{100}|U|$, a lower bound can then be given in terms of $|U|^n$.

The logarithm terms in the reduction lemma are introduced because the proof makes use of Gromov's tree approximation lemma. In their paper Delzant and Steenbock remark that if the space being acted on is a quasi-tree, then the logarithm terms in \Cref{DelzantSteenbock} should disappear (see \cite[Remark 1.18]{Delzant2020}). We are able to check that this is indeed the case. We include the proofs here for completeness, however the method and ideas are the same as those used in \cite{Delzant2020}, with the substitution of \Cref{Uniform} in place of Gromov's tree approximation lemma.

In everything that follows, we will assume that $x_0$ has been chosen to minimise $E(U)$ for our given set of isometries $U$, as in \Cref{Displacement}. For the purposes of the reduction lemma the constant $\kappa_0$ only needs to satisfy $\kappa_0\geqslant \Delta$, where $\Delta$ is the hyperbolicity or bottleneck constant for the space, however when this lemma is applied to an acylindrical action it will be taken to be the acylindricity constant from \Cref{AHDef2}.

\begin{rem}
	To reduce the number of constants, we recall from \Cref{DeltaHyperbolic} that a quasi-tree with bottleneck constant $\Delta\geqslant 0$ is $\Delta$-hyperbolic, so we will use this $\Delta$ for both.
\end{rem}

We will first consider subsets of $U$ of the form
\begin{equation*}
U_{y,z}=\{u\in U: d(x_0,ux_0)\geqslant 10^{10}\kappa_0,\ (x_0,ux_0)_y\leqslant \Delta,\ \text{and } (x_0,u^{-1}x_0)_z\leqslant \Delta\},
\end{equation*}
where $y,z\in S(x_0,1000\Delta)$, the sphere of radius $1000\Delta$ centred at $x_0$ in $X$. Note that some of these sets may be empty.

The following result is key to the proof of the reduction lemma.

\begin{lem}[Minimal energy]
	\emph{\cite{Delzant2020}}
	Let $U$ be a finite set of isometries of a $\Delta$-hyperbolic metric space $X$, with $\Delta> 0$. Let $y_0\in S(x_0,1000\Delta)$, and $Y=B(y_0,100\Delta)$ $\cap S(x_0,1000\Delta)$. Then
	\begin{equation*}
	\Bigg|\bigcup_{y,z\in Y}U_{y,z}\Bigg|\leqslant \frac{2}{3}|U|.
	\end{equation*}
\end{lem}

The idea of the minimal energy lemma is that if too many of the isometries in $U$ and $U^{-1}$ send $x_0$ in the same direction, then this would contradict our choice of $x_0$ as minimising $E(U)=\frac{1}{|U|}\sum_{u\in U}d(x,ux)$. In particular, $y_0$ would be a better choice.

The statement of the reduction lemma for quasi-trees is as follows.

\begin{lem}[Reduction lemma]
	Let $U$ be a finite set of isometries of a quasi-tree $X$ with bottleneck constant $\Delta> 0$. If at most $\frac{1}{4}$ of the isometries $u\in U$ have displacement $d(x_0,ux_0)\leqslant 10^{10}\kappa_0$, then there exist $U_0,U_1\subset U$ with cardinalities at least $\frac{1}{100}|U|$ such that
	\begin{equation*}
	(U^{-1}_0x_0,U_1x_0)_{x_0}\leqslant 1000\Delta \text{ and } (U_0x_0,U_1^{-1}x_0)_{x_0}\leqslant 1000\Delta.
	\end{equation*}
	In addition, for all $u_0\in U_0$ and $u_1\in U_1$ we have that
	\begin{equation*}
	d(x_0,u_0x_0)\geqslant 10^{10}\kappa_0 \text{ and } d(x_0,u_1x_0)\geqslant 10^{10}\kappa_0.
	\end{equation*}
\end{lem}

To prove the reduction lemma we need two preliminary lemmas, for which we will assume that $X$ and $U$ are as above. As with everything in this subsection, their proofs follow their analogues in \cite{Delzant2020}, with the only difference being the substitution of constants to match the uniform tree approximation lemma.

We first set up our notation. Let $r=1000\Delta$ and $S=S(x_0,r)$. Consider
\begin{equation*}
V=\bigcup_{x\in Ux_0\cup U^{-1}x_0}[x_0,x].
\end{equation*}
As $U$ is finite, by \Cref{Uniform} there exists an $\mathbb{R}$-tree $(T,d^*)$ and a map $f:(V,d)\to (T,d^*)$ such that:
\begin{enumerate}
	\item For all $x\in Ux_0\cup U^{-1}x_0$, the restriction of $f$ to the geodesic segment $[x_0,x]$ is an isometry.
	\item For all $y,z\in V$, we have that $d(y,z)-6\Delta\leqslant d^*(f(y),f(z))\leqslant d(y,z)$.
\end{enumerate}
Let $\overline{S}$ be the sphere of radius $r$  centred at $f(x_0)$ in $T$. As $f$ restricted to any geodesic segment $[x_0,x]$ is an isometry, we have that $\overline{S}=f(S\cap V)$. For any $P,Q\subset\overline{S}$, we let $k=10^{10}\kappa_0$ and define
\begin{align*}
\overline{U_{P,Q}}= \{ & u\in U: d(x_0,ux_0)\geqslant k,\text{ and there exists } p\in P, q\in Q \text{ such that }
\\ & (f(x_0),f(ux_0))_p=0 \text{ and } (f(x_0),f(u^{-1}x_0))_q=0\}.
\end{align*}
This is equivalent to saying that $u\in U$ with displacement $d(x_0,ux_0)\geqslant k$ is in $\overline{U_{P,Q}}$ if and only if $[f(x_0),f(ux_0)]$ intersects $\overline{S}$ at $P$, and $[f(x_0),f(u^{-1}x_0)]$ intersects $\overline{S}$ at $Q$.

The proof of \Cref{Lemma69} follows the proof of Lemma 6.9 in \cite{Delzant2020}.

\begin{lem}
	\label{Lemma69}
	Let $P\subset \overline{S}$ and $Q\subset\overline{S}\backslash P$. Then
	\begin{equation*}
	(\overline{U_{\diamond,Q}}^{-1}x_0,\overline{U_{P,\star}}x_0)_{x_0}\leqslant r,
	\end{equation*}
	where $\diamond$ and $\star$ are used as placeholders for any subsets of $\overline{S}$.
\end{lem}

\begin{proof}
	Let $u_0\in\overline{U_{P,\star}}$, and $u_1\in\overline{U_{\diamond,Q}}$. This means that $[f(x_0),f(u_0x_0)]$ intersects $\overline{S}$ at $P$, and $[f(x_0),f(u_1^{-1}x_0)]$ intersects $\overline{S}$ at $Q$. As $P\cap Q=\emptyset$, and we are in a tree, this means that
	\begin{equation*}
	(f(u_0x_0),f(u_1^{-1}x_0))^*_{f(x_0)}\leqslant r.
	\end{equation*}
	As $f$ is an isometry on $[x_0,u_0x_0]$ and $[x_0,u_1^{-1}x_0]$, and $d^*(f(u_0x_0),f(u_1^{-1}x_0))\leqslant d(u_0x_0,u_1^{-1}x_0)$, we get that 
	\begin{align*}
	(f(u_0x_0),f(u_1^{-1}x_0))^*_{f(x_0)} & =\frac{1}{2}(d^*(f(x_0),f(u_0x_0))+d^*(f(x_0),f(u_1^{-1}x_0))-d^*(f(u_0x_0),f(u_1^{-1}x_0)))
	\\ & \geqslant \frac{1}{2}(d(x_0,u_0x_0)+d(x_0,u_1^{-1}x_0)-d(u_0x_0,u_1^{-1}x_0))
	\\ & = (u_0x_0,u_1^{-1}x_0)_{x_0},
	\end{align*}
	and therefore
	\begin{equation*}
	(u_0x_0,u_1^{-1}x_0)_{x_0}\leqslant r.
	\end{equation*}
\end{proof}

Sets of the form $\overline{U_{P,\star}}$ and $\overline{U_{\diamond,Q}}$ therefore have the required Gromov product, and have the required displacement by definition. If we can find such sets that have cardinalities at least $\frac{1}{100}|U|$, then we will be done.

The proof of \Cref{Hundred} follows the proof of Lemma 6.11 from \cite{Delzant2020}. Let $P\subset \overline{S}$, and $Q=\overline{S}\backslash P$. Fix $p\in P$. Let $P'=P\backslash\{p\}$, and $Q'=Q\cup\{p\}$. Note that $|\overline{U_{P',P'}}|\leqslant|\overline{U_{P,P}}|$, and $|\overline{U_{Q',Q'}}|\geqslant|\overline{U_{Q,Q}}|$.

\begin{lem}
	\label{Hundred}
	Suppose at most $\frac{1}{4}$ of the isometries $u\in U$ have displacement $d(x_0,ux_0)\leqslant k$, and that we have the following:
	\begin{enumerate}
		\item $|\overline{U_{P,Q}}|\leqslant\frac{1}{100}|U|$, $|\overline{U_{Q,P}}|\leqslant\frac{1}{100}|U|$, and $|\overline{U_{Q,Q}}|\leqslant\frac{1}{100}|U|$.
		\item $|\overline{U_{P',Q'}}|\leqslant\frac{1}{100}|U|$, $|\overline{U_{Q',P'}}|\leqslant\frac{1}{100}|U|$, and $|\overline{U_{Q',Q'}}|>\frac{1}{100}|U|$.
	\end{enumerate}
	Then $|\overline{U_{P',P'}}|>\frac{1}{100}|U|$.
\end{lem}

\begin{proof}
	Suppose that $|\overline{U_{P',P'}}|\leqslant\frac{1}{100}|U|$. We want to get a lower bound on $|\overline{U_{p,p}}|$, and use this to get a contradiction with the minimal energy lemma.
	
	First note that as $Q=\overline{S}\backslash P$ and $T$ is a tree, the sets $\overline{U_{P,P}}$, $\overline{U_{P,Q}}$, $\overline{U_{Q,P}}$, and $\overline{U_{Q,Q}}$ are disjoint and
	\begin{equation*}
	\{u\in U: d(x_0,ux_0)\geqslant k\}=\overline{U_{P,P}}\sqcup \overline{U_{P,Q}}\sqcup \overline{U_{Q,P}}\sqcup \overline{U_{Q,Q}}.
	\end{equation*}
	We know that $|\{u\in U: d(x_0,ux_0)\geqslant k\}|\geqslant\frac{75}{100}$, so $|\overline{U_{P,P}}|\geqslant\frac{72}{100}$, and therefore $|\overline{U_{P,P}}\backslash \overline{U_{P',P'}}|\geqslant\frac{71}{100}$. We also have that
	\begin{align*}
	|\overline{U_{P,P}}\backslash \overline{U_{P',P'}}| & =|\overline{U_{P',p}}|+|\overline{U_{p,P'}}|+|\overline{U_{p,p}}|
	\\ & \leqslant |\overline{U_{P',Q'}}|+|\overline{U_{Q',P'}}|+|\overline{U_{p,p}}|
	\\ & \leqslant \frac{2}{100}|U|+|\overline{U_{p,p}}|.
	\end{align*}
	Hence $|\overline{U_{p,p}}|\geqslant\frac{69}{100}$.
	
	To get a contradiction with the minimal energy lemma, we want to show that $\overline{U_{p,p}}$ is a subset of $\bigcup_{y,z\in B(y_0,6\Delta)\cap S}U_{y,z}$ for some $y_0\in S$.
	
	Consider the set $f^{-1}(p)$. As $p\in\overline{S}$, and $f$ maps geodesics from $x_0$ isometrically, we have that $f^{-1}(p)\subset S$. Fix $y_0\in f^{-1}(p)$. For any $z\in f^{-1}(p)$, we get that $d(y_0,z)\leqslant d^*(f(y_0),f(z))+6\Delta=6\Delta$. Hence $f^{-1}(p)\subset B(y_0,6\Delta)\cap S$.
	
	It therefore suffices to show that $\overline{U_{p,p}}$ is a subset of $\bigcup_{y,z\in f^{-1}(p)}U_{y,z}$. Let $u\in\overline{U_{p,p}}$. The function $f$ maps $[x_0,ux_0]$ isometrically onto $[f(x_0),f(ux_0)]$, and $p$ lies on $[f(x_0),f(ux_0)]$, so there exists $y\in f^{-1}(p)$ such that $y\in[x_0,ux_0]$. This means that $(x_0,ux_0)_y=0<\Delta$. Similarly, we can pick $z\in f^{-1}(p)$ such that $(x_0,u^{-1}x_0)_z=0<\Delta$. Hence $u\in U_{y,z}$, so
	\begin{equation*}
	\frac{69}{100}\leqslant|\overline{U_{p,p}}|\leqslant \Bigg|\bigcup_{y,z\in f^{-1}(p)}U_{y,z}\Bigg|\leqslant\Bigg|\bigcup_{y,z\in B(y_0,6\Delta)\cap S}U_{y,z}\Bigg|.
	\end{equation*}
	This contradicts the minimal energy lemma, so $|\overline{U_{P',P'}}|>\frac{1}{100}|U|$.
\end{proof}

We will now use these results to prove the reduction lemma.

\begin{proof}[Proof of reduction lemma]
	We want to find two sets $U_0,U_1$ of the form $\overline{U_{P,\star}}$ and $\overline{U_{\diamond,Q}}$, with $P\subset \overline{S}$ and $Q\subset\overline{S}\backslash P$, such that $|\overline{U_{P,\star}}|>\frac{1}{100}$ and $|\overline{U_{\diamond,Q}}|>\frac{1}{100}$.
	
	Let $P^{(0)}=\overline{S}$ and $Q^{(0)}=\emptyset$. Note that $\overline{S}$ is finite, so we can write $P^{(0)}$ as $P^{(0)}=\{p_1,\ldots,p_N\}$ for some $N\in\mathbb{N}$. We now recursively define $P^{(n)}=P^{(n-1)}\backslash\{p_n\}$, and $Q^{(n)}=Q^{(n-1)}\cup\{p_n\}$.
	
	If for any $n$ we have that $|\overline{U_{P^{(n)},Q^{(n)}}}|>\frac{1}{100}$ or $|\overline{U_{Q^{(n)},P^{(n)}}}|>\frac{1}{100}$, then we can set $U_0=U_1$ to be this set, and we are done. Suppose otherwise, so $|\overline{U_{P^{(n)},Q^{(n)}}}|\leqslant\frac{1}{100}$ and $|\overline{U_{Q^{(n)},P^{(n)}}}|\leqslant\frac{1}{100}$ for all $n$. We can see that $|\overline{U_{Q^{(n)},Q^{(n)}}}|$ is an increasing sequence, with $|\overline{U_{Q^{(0)},Q^{(0)}}}|=0$ and
	\begin{equation*} |\overline{U_{Q^{(N)},Q^{(N)}}}|=|\overline{U_{\overline{S},\overline{S}}}|=|\{u\in U:d(x_0,ux_0)\geqslant k\}|\geqslant\frac{75}{100}.
	\end{equation*}
	Hence there exists $n$ such that $|\overline{U_{Q^{(n-1)},Q^{(n-1)}}}|\leqslant\frac{1}{100}$ and $|\overline{U_{Q^{(n)},Q^{(n)}}}|>\frac{1}{100}$. By \Cref{Hundred} we also get that $|\overline{U_{P^{(n)},P^{(n)}}}|>\frac{1}{100}$, so we can set $U_0=\overline{U_{P^{(n)},P^{(n)}}}$ and $U_1=\overline{U_{Q^{(n)},Q^{(n)}}}$.
\end{proof}

The subsets used in \cite{Delzant2020} actually take a slightly stronger form, where the displacement of $x_0$ by one subset is greater than that of the other, however as in \cite{Delzant2020} such sets are easily obtainable as a corollary of the reduction lemma.

\begin{cor}
	\label{leq100}
	Let $U$ be a finite set of isometries of a quasi-tree $X$ with bottleneck constant $\Delta> 0$. If at most $\frac{1}{4}$ of the isometries $u\in U$ have displacement $d(x_0,ux_0)\leqslant 10^{10}\kappa_0$, then there exist $U_0,U_1\subset U$ with cardinalities at least $\frac{1}{200}|U|$ such that
	\begin{equation*}
	(U^{-1}_0x_0,U_1x_0)_{x_0}\leqslant 1000\Delta \text{ and } (U_0x_0,U_1^{-1}x_0)_{x_0}\leqslant 1000\Delta.
	\end{equation*}
	In addition, for all $u_0\in U_0$ and $u_1\in U_1$ we have that
	\begin{equation*}
	10^{10}\kappa_0\leqslant d(x_0,u_0x_0)\leqslant d(x_0,u_1x_0).
	\end{equation*}
\end{cor}

\begin{proof}
	Let $U_0'$ and $U_1'$ be the sets chosen in the reduction lemma. Let $m_0$ be the median of $\{d(x_0,ux_0):u\in U_0'\}$, and $m_1$ be the median of $\{d(x_0,ux_0):u\in U_1'\}$.
	
	If $m_0\leqslant m_1$, then let $U_0=\{u\in U_0':d(x_0,ux_0)\leqslant m_0\}$ and $U_1=\{u\in U_1':d(x_0,ux_0)\geqslant m_1\}$. If $m_1<m_0$ then let $U_0=\{u\in U_1':d(x_0,ux_0)\leqslant m_1\}$ and $U_1=\{u\in U_0':d(x_0,ux_0)\geqslant m_0\}$. In both cases note that $U_0$ and $U_1$ have cardinalities at least half of that of the original sets.
\end{proof}

\subsubsection{Growth of sets with large displacement}

When combined with the rest of the work in \cite{Delzant2020}, \Cref{leq100} gives us the desired result for acylindrical actions on quasi-trees.

\begin{thm}
	\label{KChoice}
	Let $G$ be a group acting $(\kappa_0,N_0)$-acylindrically on a quasi-tree $X$ with bottleneck constant $\Delta> 0$. Assume $\kappa_0\geqslant\Delta$ and $N_0\geqslant 1$. There exist constants $\alpha=\alpha(\Delta,\kappa_0,N_0)>0$ and $K=K(\kappa_0)>0$ such that for every finite $U\subset G$, at least one of the following must hold:
	\begin{enumerate}
		\item $\langle U\rangle$ is virtually $\mathbb{Z}$.
		\item $\lambda_0(U)< K$.
		\item $|U^n|\geqslant (\alpha|U|)^{\lfloor\frac{n+1}{2}\rfloor}$ for every $n\in\mathbb{N}$.
	\end{enumerate}
	In particular, we can take $\alpha=\frac{\Delta^2}{10^{52}N_0^6\kappa_0^2}$ and $K=10^{14}\kappa_0$.
\end{thm}

\begin{proof}
	This follows by combining \Cref{leq100} from this paper with Corollary 5.6 and Proposition 6.18 from \cite{Delzant2020}, using $d=1$, $b=10$, and $\delta=\Delta$.
\end{proof}

\begin{rem}
	This is a generalisation of Delzant and Steenbock's result for groups acting acylindrically on trees \cite[Theorem 1.11]{Delzant2020}.
\end{rem}

Recall \Cref{Balasubramanya}, which states that every acylindrically hyperbolic group admits a non-elementary acylindrical action on a quasi-tree. This allows us to apply \Cref{KChoice} to the entire class of acylindrically hyperbolic groups.

We emphasise here, however, that this is not a direct improvement of \Cref{DelzantSteenbock}, as in both statements the displacement condition is dependent on the action under consideration. In particular, for a certain acylindrically hyperbolic group $G$, and finite subset $U$ of $G$, we may have that the displacement of $U$ is large under the acylindrical action of $G$ on some general hyperbolic space, but small under the acylindrical action of $G$ on a quasi-tree.



\subsubsection{Loxodromic elements and displacement}

To use \Cref{KChoice} to say anything about the growth of finite subsets of a specific acylindrically hyperbolic group, we must first be able to say something about which finite subsets will have large displacement under some action on a quasi-tree.

One difficulty with this is that the displacement $\lambda_0$ is dependent on the choice of basepoint $x_0$, which is dependent on the finite set $U$. It is easier instead to try to find sets $U$ such that $\max_{u\in U}d(x,ux)$ is large for every point $x$ in the space being acted on. As is often the case with exponential growth questions, this can be done by finding loxodromic elements, due to the following result of Bowditch.

\begin{defn}
	Let $G$ be a group acting by isometries on a metric space $X$. For $g\in G$, the \emph{stable translation length} of $g$ is
	\begin{equation*}
		\tau(g)=\lim_{n\to\infty}\frac{d(x,g^nx)}{n},
		\end{equation*}
	where $x\in X$ is arbitrary.
\end{defn}

\begin{prop}
	\label{LowerBound}
	\emph{\cite{Bowditch2008}}
	Let $G$ be a group acting acylindrically on a $\delta$-hyperbolic space. There exists $\nu>0$, dependent only on $\delta$ and the acylindricity constants, such that if $g\in G$ is loxodromic then $\tau(g)\geqslant \nu$.
\end{prop}

We also need the following lemma.

\begin{lem}
	\label{NoWeirdness}
	Let $G$ be a group acting acylindrically on a hyperbolic space $X$, and let $U\subset G$ be finite. Suppose that $U^k$ contains a loxodromic element for some $k\in\mathbb{N}$, and $\langle U\rangle$ is not virtually $\mathbb{Z}$. Then for all $n\in\mathbb{N}$, we have that $\langle U^n\rangle$ is not virtually $\mathbb{Z}$.
\end{lem}

\begin{proof}
	Let $g\in U^k$ be loxodromic. Suppose that for some $n\in\mathbb{N}$, we have that $\langle U^n\rangle$ is virtually $\mathbb{Z}$. As $g^n\in U^{kn}$, we have that $\langle U^n\rangle\leqslant E(g)$, so in particular $U^n\{g^+,g^-\}=\{g^+,g^-\}$.
	
	We want to show that for every $u\in U$ and $1\leqslant m\leqslant n$, we have that $U^m\{g^+,g^-\}=u^m\{g^+,g^-\}$. We know that this is true when $m=n$. Suppose true of some $2\leqslant m\leqslant n$. Let $u\in U$, then $uU^{m-1}\{g^+,g^-\}\subset U^m\{g^+,g^-\}=u^m\{g^+,g^-\}$. As $uU^{m-1}\{g^+,g^-\}$ is nonempty we must have equality, so $U^{m-1}\{g^+,g^-\}=u^{m-1}\{g^+,g^-\}$. This proves the claim by induction.
	
	We therefore have that $A=\{\{g^+,g^-\},u\{g^+,g^-\},\ldots,u^{n-1}\{g^+,g^-\}\}$ does not depend on the choice of $u\in U$, and as $u^n\in E(g)$ we have that $A$ is closed under applying any element of $U$. Given that for any $u\in \mathbb{N}$ we can rewrite $A$ as $\{u\{g^+,g^-\},\ldots,u^{n}\{g^+,g^-\}\}$, we also get that $A$ is closed under applying any element of $U^{-1}$. This means that $A$ is the orbit of $\{g^+,g^-\}$ under the action of $\langle U\rangle$. However $A$ is finite, and $\langle U\rangle$ has a non-elementary acylindrical action on $X$, so this is a contradiction by \cite[8.2.D]{Gromov1987}.
\end{proof}

The above two results can then be applied to \Cref{KChoice}, the idea of which was originally pointed out by Thomas Delzant.

\begin{prop}
	\label{ShortLoxodromic}
	Let $G$ be a group acting acylindrically on a quasi-tree. There exist $\alpha,\beta>0$ such that for every finite $U\subset G$ such that $U$ contains a loxodromic element, and $\langle U\rangle$ is not virtually $\mathbb{Z}$, we have that $|U^n|\geqslant (\alpha|U|)^{\beta n}$ for every $n\in\mathbb{N}$.
\end{prop}

\begin{proof}
	Let $X$ be the quasi-tree on which $G$ acts acylindrically. Let $u\in U$ be our given loxodromic element. By \Cref{LowerBound} we know that there exists $\nu>0$, dependent only on the bottleneck constant of $X$ and the acylindricity constants, such that $\tau(u)\geqslant \nu$. Let $K>0$ be as in \Cref{KChoice}, and let $m=\lceil \frac{K}{\nu}\rceil$. It is clear from the definition of the stable translation length that $\tau(u^m)=m\tau(u)\geqslant K$.
	
	Let $x\in X$ be arbitrary, and note that for $n\in\mathbb{N}$ we have that $d(x,u^{nm}x)\leqslant nd(x,u^mx)$, so $\tau(u^m)\leqslant d(x,u^mx)$. Therefore, as $u^m\in U^{m}$, we conclude that $\lambda_0(U^{m})\geqslant K$. By \Cref{NoWeirdness} we have that $\langle U^m\rangle$ is not virtually $\mathbb{Z}$, so we can apply \Cref{KChoice} to say that
	\begin{equation*}
		|U^{nm}|\geqslant (\alpha|U^m|)^{\lfloor\frac{n+1}{2}\rfloor}\geqslant (\alpha|U|)^{\lfloor\frac{n+1}{2}\rfloor},
	\end{equation*}
	for all $n\in\mathbb{N}$, where $\alpha>0$ chosen in such a way that it is dependent only on the bottleneck constant of $X$ and the acylindricity constants of the action.
	
	Suppose $i\in\mathbb{N}$ is such that $i\geqslant m$. Then there exists $n\in\mathbb{N}$ such that $nm\leqslant i<(n+1)m$. Therefore
	\begin{equation*}
		|U^i|\geqslant|U^{nm}|\geqslant (\alpha|U|)^{\lfloor\frac{n+1}{2}\rfloor}\geqslant (\alpha|U|)^{\frac{i}{4m}}.
	\end{equation*}
	Now suppose $i<m$. If $\alpha|U|<1$ then it is trivial that $|U^i|\geqslant (\alpha|U|)^{\frac{i}{4m}}$, as $|U^i|\geqslant 1$. If we instead have that $\alpha|U|\geqslant1$, then \Cref{KChoice} tells us that $\alpha<1$, so $|U^i|\geqslant |U|\geqslant \alpha |U|\geqslant(\alpha|U|)^{\frac{i}{4m}}$, as $\frac{i}{4m}<1$.
\end{proof}

\begin{rem}
	The constants $\alpha$ and $\beta$ are dependent on the acylindricity constants of the action, and the bottleneck constant of the space being acted on, so two different acylindrical actions on quasi-trees by the same group will give different constants.
\end{rem}

\begin{cor}
	\label{ShortLoxodromicCor}
	Let $G$ be a group acting acylindrically on a quasi-tree. There exist $\alpha,\beta>0$ such that for every finite $U\subset G$ such that $U^k$ contains a loxodromic element for some $k\in\mathbb{N}$, and $\langle U\rangle$ is not virtually $\mathbb{Z}$, we have that $|U^n|\geqslant (\alpha|U|)^{\frac{\beta n}{k}}$ for every $n\in\mathbb{N}$.
\end{cor}

\begin{proof}
	By \Cref{NoWeirdness}, the set $U^k$ satisfies the conditions of \Cref{ShortLoxodromic}, so we have that
	\begin{equation*}
		|U^n|\geqslant |U^{nk}|^{\frac{1}{k}}\geqslant (\alpha|U^k|)^{\frac{\beta n}{k}}\geqslant (\alpha|U|)^{\frac{\beta n}{k}}.
	\end{equation*}
\end{proof}

The growth we get for a certain subset of our group is therefore dependent on $k$, however if we can find an upper bound on $k$, then we will have a lower bound for the growth. In other words, given a class of finite subsets of $G$, if we can always generate a loxodromic element on a quasi-tree within a bounded product of each set, then every $U$ in this class will satisfy $|U^n|\geqslant (\alpha|U|)^{\beta n}$ for some uniform constants $\alpha,\beta>0$.

We note here that these loxodromic elements do not have to be obtained from a single acylindrical action on a quasi-tree. Instead, we only require that the quasi-trees and associated acylindrical actions in question admit an upper bound on the bottleneck constants and acylindricity constants. In particular, this will automatically be satisfied if the loxodromic elements can be found from a finite collection of quasi-trees and acylindrical actions, as then the associated constants will automatically be bounded.

\begin{rem}
	\label{TreeAction}
	If $G$ has a non-elementary acylindrical action on a simplicial tree, without edge inversions, then the conditions of \Cref{ShortLoxodromicCor} are automatically satisfied for $k=2$ \cite[p.~64]{Serre1980}. This also follows from \cite[Theorem 1.11]{Delzant2020}.
\end{rem}




\subsection{Growth from actions on hyperbolic spaces}

In the previous section we showed that being able to quickly generate loxodromic elements in an acylindrical action on a quasi-tree is enough to prove uniform product set growth. \Cref{Balasubramanya} tells us that every acylindrically hyperbolic group admits an acylindrical action on a quasi-tree, however this construction does not necessarily give us an easy way to find loxodromic elements. There are groups with more natural actions on quasi-trees where this is possible to do, as we will see in Section 4, however for many acylindrically hyperbolic groups their most natural acylindrical action is on a hyperbolic space that is not a quasi-tree.

In such a case, we can instead use a result of Fujiwara, who showed that the statements we proved in Subsection 3.2.3 also hold for finite symmetric subsets of groups that act acylindrically on hyperbolic spaces, using the idea of generating free subgroups with rank proportional to the size of the initial set  \cite{Fujiwara2021}. Note that this does not replace the generality of \Cref{KChoice}, as Fujiwara's result only considers symmetric subsets, and does not allow for other possible ways of finding sets of large displacement. On the other hand, it does give us a practical way of answering \Cref{MainQuestion} for a wider variety of acylindrically hyperbolic groups.

\begin{prop}
	\emph{\cite{Fujiwara2021}}
	\label{FreeSubgroup}
	Let $G$ be a group acting acylindrically on a hyperbolic space. There exist $a,b\in\mathbb{N}$ such that for every finite symmetric $U\subset G$ such that $U^k$ contains a loxodromic element for some $k\in\mathbb{N}$, and $\langle U\rangle$ is not virtually $\mathbb{Z}$, we have that there exists $F\subset U^{bk}$ such that $|F|\geqslant \frac{|U|}{a}$ and $F$ generates a free subgroup of rank $|F|$.
\end{prop}

This allowed Fujiwara to obtain an alternative version of \Cref{ShortLoxodromicCor}, for finite symmetric subsets of groups with acylindrical actions on hyperbolic spaces.

\begin{cor}
	\emph{\cite{Fujiwara2021}}
	\label{HypLoxCor}
	Let $G$ be a group acting acylindrically on a hyperbolic space. There exist $\alpha,\beta>0$ such that for every finite symmetric $U\subset G$ such that $U^k$ contains a loxodromic element for some $k\in\mathbb{N}$, and $\langle U\rangle$ is not virtually $\mathbb{Z}$, we have that $|U^n|\geqslant (\alpha|U|)^{\frac{\beta n}{k}}$ for every $n\in\mathbb{N}$.
\end{cor}

\begin{proof}
	By \Cref{FreeSubgroup}, there exist $a,b\in\mathbb{N}$ that are not dependent on $U$ such that we can find $F\subset U^{bk}$, where $|F|\geqslant \frac{|U|}{a}$ and $F$ generates a free subgroup of rank $|F|$. This gives us that
	\begin{equation*}
		|U^n|\geqslant |U^{nbk}|^{\frac{1}{bk}}\geqslant |F^n|^{\frac{1}{bk}}=|F|^{\frac{n}{bk}} \geqslant \bigg(\frac{1}{a}|U|\bigg)^{\frac{n}{bk}}.
	\end{equation*}
	Letting $\alpha=\frac{1}{a}$ and $\beta=\frac{1}{b}$ completes the proof.
\end{proof}

As in the quasi-tree case, this means that if we have a collection of subgroups that each act acylindrically on one of a collection of hyperbolic spaces, such that all the constants involved are bounded, and for each symmetric generating set of a subgroup we can quickly generate a loxodromic element in one of these actions, then this will give us our desired uniform product set growth.

\section{Applications to virtual subgroups of mapping class groups}

In this section we will apply the results of Sections 2 and 3 to answer \Cref{MainQuestion} for virtual subgroups of mapping class groups. That is, we would like to find a dichotomy for the finitely generated virtual subgroups of our mapping class group, where either $|U^n|\geqslant (\alpha|U|)^{\beta n}$ for every symmetric generating set of our subgroup, with $\alpha,\beta>0$ dependent only on the mapping class group in question, or our subgroup cannot satisfy this property for any $\alpha,\beta>0$.

Answering this question for mapping class groups also answers it for any group that embeds as a subgroup of a mapping class group, which includes right-angled Artin groups. We will however deal with the right-angled Artin group case separately here. The reason for doing this is that right-angled Artin groups have a natural acylindrical action on an associated quasi-tree called the extension graph, and so in proving something about the growth of these groups we have to prove that we can quickly generate loxodromic elements on these extension graphs. This is analogous to an already known result regarding the action of mapping class groups on curve complexes \cite{Mangahas2013}. In addition, the existence of these loxodromic elements has an application to the set of exponential growth rates of a right-angled Artin group.

The other reason for treating right-angled Artin group separately is that, as they naturally act acylindrically on quasi-trees, the only reason we need to restrict ourselves to symmetric subsets is because of the method we use to find loxodromic elements. If a non-symmetric way of doing this were to be found, then a non-symmetric version of uniform product set growth would automatically follow from the results in this paper.

In Section 4.1, we will recall some the necessary background information about right-angled Artin groups and mapping class groups. In Section 4.2, we will prove our result regarding the uniformly quick generation of loxodromic elements in right-angled Artin groups, and in Section 4.3 we use this to obtain a product set growth result. We then finish by extending this result to mapping class groups in Section 4.4.

\subsection{Background on right-angled Artin groups and mapping class groups}

\subsubsection{Right-angled Artin groups}

\begin{defn}
	For a finite simple graph $\Gamma=(V,E)$, its \emph{right-angled Artin group} is
	\begin{equation*}
		A(\Gamma)=\langle\ V\ |\ [v,w]=1\text{ if }\{v,w\}\in E\ \rangle.
	\end{equation*}
\end{defn}

\begin{exm}
	If $\Gamma$ does not contain any edges, then $A(\Gamma)$ will be the free group of rank $|V|$. If $\Gamma$ is a complete graph, then $A(\Gamma)$ is the free abelian group of rank $|V|$.
\end{exm}

\begin{rem}
	Right-angled Artin groups are always torsion-free.
\end{rem}

Every right-angled Artin group has another associated graph called the extension graph, denoted $\Gamma^e$ (for a definition see \cite{Kim2013}). So long as the defining graph $\Gamma$ is connected, we get the following result.

\begin{thm}
	\emph{\cite{Kim2013,Kim2014}}
	Let $\Gamma$ be a finite connected graph. The extension graph $\Gamma^e$ is a quasi-tree, and the action of $A(\Gamma)$ on $\Gamma^e$ is acylindrical.
\end{thm}

In Section 4.2 we will be interested in generating loxodromic elements in the action of $A(\Gamma)$ on $\Gamma^e$. This is made easier by the fact that there exists a characterisation of these elements, given in \cite{Kim2014}. To state this characterisation, we first need to give some terminology.

\begin{rem}
	With the exception of the extension graph, all other graphs in the context of right-angled Artin groups will be assumed to be finite and simple.
\end{rem}

\begin{note}
	We denote the vertex set of a graph $\Gamma$ by $V$, and the edge set by $E$. We denote the complement of a graph $\Gamma$ by $\Gamma^c$. If $V'\subset V$, we denote the subgraph of $\Gamma$ induced by $V'$ as $\Gamma(V')$.
\end{note}

\begin{defn}
	The \emph{join} of two graphs $\Gamma_1$ and $\Gamma_2$ is $\Gamma_1 \ast \Gamma_2=(\Gamma_1^c \sqcup \Gamma_2^c)^c$.
\end{defn}

Alternatively, we can see that if $\Gamma_1=(V_1,E_1)$ and $\Gamma_2=(V_2,E_2)$ then $\Gamma_1\ast \Gamma_2=(V_1\cup V_2,E)$, where $E=E_1\cup E_2\cup\{\{s_1,s_2\}:s_1\in V_1\text{ and }s_2\in V_2\}$.

\begin{defn}
	A graph $\Gamma$ is said to \emph{split as a non-trivial join} if $\Gamma=\Gamma_1\ast\Gamma_2$, where $\Gamma_1$ and $\Gamma_2$ are nonempty graphs. A \emph{subjoin} of a graph is an induced subgraph that splits as a nontrivial join.
\end{defn}

\begin{exm}
	If $\Gamma=\Gamma_1\ast \Gamma_2$, then $A(\Gamma)=A(\Gamma_1)\times A(\Gamma_2)$.
\end{exm}

Let $\Gamma=(V,E)$ be a graph. Recall that $A(\Gamma)$ is generated by $V$, so $g\in A(\Gamma)$ can be represented by a word $w$ written in the alphabet $V\cup V^{-1}$.

\begin{defn}
	A word $w$ representing $g\in A(\Gamma)$ in the alphabet $V\cup V^{-1}$ is \emph{reduced} if its length is the same as the word length of $g$. In other words, it is a minimal length representative of $g$. We will often identify $g$ with a reduced word representing it, and say that $g$ is reduced.
	
	The word $w$ is \emph{cyclically reduced} if every cyclic permutation of $w$ is also reduced. Equivalently, it is a minimal length representative of the conjugacy class of $g$.
\end{defn}

\begin{rem}
	\label{ReducedRemark}
	This is very similar to the notions of reduced and cyclically reduced words in a free group. The difference is that in the free group every element has a unique reduced word representing it, whereas in a right-angled Artin group each reduced representative is only unique up to being able to swap the order of vertices when they commute. See Section 2 of \cite{Antolin2015} for a more detailed explanation of this.
\end{rem}

\begin{rem}
	\label{CyclicallyReducedRemark}
	Every reduced word $w$ can be written in the form $w=uw'u^{-1}$, where $w'$ is a cyclically reduced word, and $u$ is a possibly empty word. This $w'$ is a minimal length representative of the conjugation class of $w$.
\end{rem}

\begin{defn}
	The \emph{support} of a word $w$ is the set of vertices $s\in V$ such that $s$ or $s^{-1}$ is a letter of $w$. The support of $g\in A(\Gamma)$, denoted $\text{supp}(g)$, is the support of a reduced word that represents $g$.
\end{defn}

\begin{rem}
	By \Cref{ReducedRemark}, the support of $g\in A(\Gamma)$ does not depend on the reduced word chosen.
\end{rem}

The following terminology is defined more generally for graph products, see \cite{Antolin2015} for example, however here we will only be using it for right-angled Artin groups.

\begin{defn}
	For $U\subset A(\Gamma)$, the \emph{essential support} of $U$, denoted $\text{esupp}(U)\subset V$, is the minimal set $V'\subset V$ under inclusion such that $U$ is conjugate into $A(\Gamma(V'))$.
\end{defn}

This is well-defined by \cite[Proposition 2.6]{Duncan2007}.

\begin{rem}
	\label{EssentialCyclicallyReduced}
	We can note that, for $g\in A(\Gamma)$, $\text{esupp}(g)$ coincides with $\text{supp}(g')$, where $g'$ is as in \Cref{CyclicallyReducedRemark}.
\end{rem}

The characterisation of elliptic and loxodromic elements in the action of $A(\Gamma)$ on $\Gamma^e$ is then as follows.

\begin{thm}
	\emph{\cite{Kim2014}}
	\label{LoxChar}
	Let $\Gamma$ be a connected graph that is not an isolated vertex, and suppose $1\neq g\in A(\Gamma)$. Then $g$ is elliptic in the action of $A(\Gamma)$ on $\Gamma^e$ if and only if $\emph{esupp}(g)$ is contained in a subjoin of $\Gamma$.
\end{thm}

\subsubsection{Mapping class groups}

We now move on to giving some basic definitions and facts about mapping class groups.

\begin{defn}
	Let $S$ be an oriented surface with finite genus and finitely many boundary components, punctures, and connected components. The \emph{mapping class group} of $S$, denoted $MCG(S)$, is the group of orientation preserving isotopy classes of homeomorphisms of $S$ that restrict to the identity on the boundary $\partial S$, where the isotopies fix components of the boundary pointwise. These isotopy classes are called \emph{mapping classes}.
\end{defn}

\begin{con}
	All surfaces in this paper will have finite genus and finitely many boundary components, punctures, and connected components, which is known as being of \emph{finite type}. We will also assume that every surface in this paper is oriented.
\end{con}

\begin{defn}
	The \emph{sporadic} surfaces are the sphere with up to four punctures, and the torus with up to one puncture.
\end{defn}

\begin{rem}
	The mapping class group of a sporadic surface is hyperbolic \cite{Farb2011,Osin2016}.
\end{rem}

On the other hand, most mapping class groups are not hyperbolic, however, as in the right-angled Artin group case, it turns out that many mapping class groups have an acylindrical action on an associated hyperbolic space.

\begin{thm}
	\emph{\cite{Masur1999,Bowditch2008}}
	\label{CurveComplex}
	Let $S$ be a non-sporadic connected surface without boundary. The mapping class group $MCG(S)$ acts acylindrically on the curve complex associated to the surface $S$, which is a hyperbolic space.
\end{thm}

\begin{rem}
	The hyperbolicity constant of the curve complex is independent of $S$ \cite{Aougab2012,Clay2014,Hensel2015}.
\end{rem}


\begin{rem}
	\label{NotAH}
	The mapping class group of a surface with non-empty boundary would have an infinite finitely generated centre \cite{Farb2011}, and so would not have uniform product set growth by \Cref{InfFICentre}. It would also not be acylindrically hyperbolic \cite{Osin2016}, and partially for this reason the mapping class group is sometimes defined either for surfaces without boundary, or the boundary is fixed setwise rather than pointwise. In the latter case the boundary components can effectively be seen as punctures that are not allowed to permute, and this version of the mapping class group is a finite index subgroup of the version where the punctures are allowed to permute.
	Although in our definition we fix our boundary pointwise, we will later see that we can easily reduce to the case where we only consider punctures.
\end{rem}

The curve complex in \Cref{CurveComplex} is defined in many places, including \cite{Masur1999}, where its hyperbolicity was first proved. As with the extension graph for right-angled Artin groups, we will not need the exact definition of it here. The important facts are the ones given by the above theorem, and that, again as with the extension graph, we have a characterisation of the loxodromic elements in this action.

\begin{defn}
	A simple closed curve on a surface $S$ is \emph{essential} if it is not homotopic to a point, a puncture, or to any component of the boundary $\partial S$.
\end{defn}

\begin{defn}
	A mapping class $f\in MCG(S)$ is a \emph{pseudo-Anosov} if there is no isotopy class of an essential simple closed curve in $S$ that is fixed by a power of $f$. A mapping class $f$ is said to be \emph{pseudo-Anosov on a connected subsurface} $S'$ if there exists a representative of $f$ such that its restriction to $S'$ gives a pseudo-Anosov mapping class in $MCG(S')$.
\end{defn}

\begin{prop}
	\emph{\cite{Masur1999}}
	Let $S$ be a non-sporadic connected surface without boundary. A mapping class $f\in MCG(S)$ is a loxodromic element in the acylindrical action on the curve complex associated to $S$ if and only if it is a pseudo-Anosov.
\end{prop}

When we consider mapping class groups, we will therefore not talk about loxodromic elements, but refer to pseudo-Anosovs instead. In other words, to apply our results from Section 3 to mapping class groups, we will need the ability to quickly generate pseudo-Anosovs, which is provided by the following theorem.

\begin{thm}
	\emph{\cite{Mangahas2013}}
	\label{ShortPA1}
	Let $S$ be a non-sporadic connected surface without boundary. There exists a constant $N=N(S)\in\mathbb{N}$ such that for any finite symmetric $U\subset MCG(S)$, where $\langle U\rangle$ contains a pseudo-Anosov, there exists $n\leqslant N$ such that $U^n$ contains a pseudo-Anosov.
\end{thm}

The problem with using the above statement to show that subgroups of mapping class groups have uniform product set growth is that many interesting subgroups may not contain a pseudo-Anosov on the whole surface. The standard way to deal with this situation is to cut our surface in some canonical way, and consider the groups induced by the restrictions of our mapping classes to the remaining subsurfaces.

\begin{prop}[Cutting homomorphism]
	\emph{\cite[Proposition 3.20]{Farb2011}}
	\label{Cutting}
	Let $S$ be a surface, and let $\{\gamma_1,\ldots,\gamma_k\}$ be a set of disjoint and isotopically distinct essential simple closed curves in $S$. We denote by $MCG(S,\{\gamma_1,\ldots,\gamma_n\})$ the subgroup of $MCG(S)$ containing the mapping classes that fix this set of curves. Then there exists a natural homomorphism $\varphi:MCG(S,\{\gamma_1,\ldots,\gamma_n\})\to MCG(S\backslash\bigcup_{i=1}^k\gamma_i)$, with kernel $\langle T_{\gamma_1},\ldots, T_{\gamma_k}\rangle$, where $T_{\gamma_i}$ is a Dehn twist around $\gamma_i$.
\end{prop}

We refer elsewhere for the definition of a Dehn twist around a curve $\gamma$, for example \cite{Farb2011}. The important facts for us to know here about Dehn twists are that they are infinite order elements, and are central in any group that fixes $\gamma$.

\begin{lem}
	\emph{\cite[Fact 3.7]{Farb2011}}
	\label{DehnConjugate}
	Let $\gamma$ be a simple closed curve in a surface $S$. Let $T_{\gamma}$ be a Dehn twist around $\gamma$, then for every $g\in MCG(S)$ we have that $gT_{\gamma}g^{-1}=T_{g(\gamma)}$.
\end{lem}

\begin{cor}
	\emph{\cite[Fact 3.8]{Farb2011}}
	\label{DehnCommute}
	Let $\gamma$ be a simple closed curve in a surface $S$. The Dehn twist $T_{\gamma}$ commutes with every mapping class in $MCG(S,\{\gamma\})$.
\end{cor}

In particular, Dehn twists around disjoint and homotopically distinct simple closed curves commute with each other, so in \Cref{Cutting} we have that $\langle T_{\gamma_1},\ldots, T_{\gamma_k}\rangle\cong \mathbb{Z}^k$.

Recall from Section 2.2.1 that if a group has an infinite order element in its centre, then it does not have uniform product set growth. This means that if we work with subgroups of $MCG(S)$ that fix individual curves (rather than possibly permute a set of them), then ruling out the subgroups that have such infinite order central elements will mean that the restriction of $\varphi$ in \Cref{Cutting} will give us an injective homomorphism into $MCG(S\backslash\bigcup_{i=1}^k\gamma_i)$. If we further assume that our subgroup does not permute the connected components of $S\backslash\bigcup_{i=1}^k\gamma_i$, then this will be a map into a direct product, which Section 2.2 tells us how to deal with.

For these reasons, it makes sense for us to temporarily restrict ourselves to thinking about the pure subgroups of mapping class groups.

\begin{defn}
	A mapping class $f\in MCG(S)$ is said to be \emph{pure} if there exists a set $\{\gamma_1,\ldots,\gamma_k\}$ of disjoint and isotopically distinct essential simple closed curves in $S$ such that $f$ fixes each $\gamma_i$, does not permute the connected components of $S\backslash \bigcup_{i=1}^k\gamma_i$, and there exists a representative of $f$ such that the restriction of this representative to any connected component of $S\backslash \bigcup_{i=1}^k\gamma_i$ is either pseudo-Anosov or the identity. A subgroup of $MCG(S)$ is pure if every element in that subgroup is a pure mapping class.
\end{defn}

Not only are pure subgroups easier to work with, it turns out that every subgroup of mapping class group has a pure subgroup of finite index, by the following theorem.

\begin{thm}
	\emph{\cite{Ivanov1992}}
	\label{FiniteIndex}
	Every mapping class group has a pure normal subgroup of finite index.
\end{thm}

A common approach to studying subgroups is therefore to first work with the pure subgroups, and then extend what we find to their supergroups. This is the approach we will take when trying to prove uniform product set growth for certain subgroups of mapping class groups, with the extension to supergroups allowed by \Cref{Supergroup}. Specifically, given a subgroup of a mapping class group, we can take a pure finite index subgroup, cut along a set of curves that are fixed by this pure subgroup, and consider the restriction to subsurfaces on which our group is not the identity.

\begin{defn}
	For a pure subgroup $G\leqslant MCG(S)$, the \emph{canonical reduction multicurve} $\sigma$ is the (possibly empty) union of all isotopically distinct essential simple closed curves $\gamma$ in $S$ such that $G$ fixes $\gamma$, but if any other essential simple closed curve $\xi$ intersects $\gamma$ then $G$ does not fix $\xi$. 
\end{defn}

\begin{thm}
	\emph{\cite{Ivanov1992}}
	\label{PureCutting}
	Let $G\leqslant MCG(S)$ be a pure subgroup of a mapping class group. Let $\{\gamma_1,\ldots,\gamma_k\}$ be the set of curves in the canonical reduction multicurve for $G$. Let $\mathcal{S}$ be the set of connected components of $S\backslash\bigcup_{i=1}^k\gamma_i$. Then the homomorphism in \Cref{Cutting} restricts to a homomorphism $\varphi: G\to \Pi_{\Sigma\in \mathcal{S}}MCG(\Sigma)$, with kernel $\langle T_{\gamma_1},\ldots, T_{\gamma_k}\rangle\cap G$. Moreover, the projection of $G$ to any factor $MCG(\Sigma)$ is either trivial, or contains a pseudo-Anosov on $\Sigma$.
\end{thm}

\begin{rem}
	\label{NormalRemark}
	It follows from \Cref{DehnCommute} that $\langle T_{\gamma_1},\ldots, T_{\gamma_k}\rangle\cap G$ is central in $G$, as $g(\gamma_i)=\gamma_i$ for every $g\in G$.
\end{rem}

To apply our results from Section 2.2.3 to the direct product in \Cref{PureCutting}, we need a bound on the number of factors. The following result allows us to do this.

\begin{thm}
	\emph{\cite{Birman1983}}
	\label{FreeAbelian}
	Let $S$ be a surface of genus $g$ with $p$ punctures, and $c$ connected components. A free abelian subgroup of $MCG(G)$ has rank at most $3g+p-3c$.
\end{thm}

Mangahas' result regarding the quick generation of pseudo-Anosovs (see \Cref{ShortPA1}) was in fact a particular case of a more general result that she proved. The rough idea is that if we take the union of all subsurfaces on which some element of our group acts as a pseudo-Anosov, then we can find an element that acts as a pseudo-Anosov on this entire union within some bounded product of our generating set. The language defined here will also be useful in our work on right-angled Artin groups.

\begin{defn}
	The \emph{active subsurface} $\mathcal{A}(G)$ of a pure subgroup $G\leqslant MCG(S)$, with canonical reduction multicurve $\sigma$, is the union of the connected components of $S\backslash \sigma$ such that some mapping class in $G$ is pseudo-Anosov on that connected component, plus the annuli that are the neighbourhood of any $\gamma\in\sigma$ that are not in the boundary of an already selected component. The active subsurface of a pure mapping class $f$ is the active subsurface of $\langle f\rangle$.
\end{defn}

The fact that every subgroup of a mapping class group has a pure finite index subgroup allows the extension of the definition of an active subsurface to all subgroups of the mapping class group.

\begin{defn}
	The \emph{active subsurface} $\mathcal{A}(G)$ of an arbitrary subgroup $G\leqslant MCG(S)$, with pure finite index subgroup $H$, is given by $\mathcal{A}(G)=\mathcal{A}(H)$.
\end{defn}

This subsurface does not depend on the choice of finite index subgroup \cite{Ivanov1992}. With these definitions we can now state Mangahas' full result.

\begin{thm}
	\emph{\cite{Mangahas2013}}
	\label{ShortPA}
	Let $S$ be a non-sporadic connected surface without boundary. There exists a constant $N=N(S)\in\mathbb{N}$ such that for any finite symmetric $U\subset MCG(S)$, there exists $n\leqslant N$ and $f\in U^n$ such that $f$ has the same active subsurface as $\langle U\rangle$.
\end{thm}

Many of the tools that we have introduced here require that the surface in question has no boundary. As noted in \Cref{NotAH}, if the surface had nonempty boundary then the mapping class group would have infinite centre, namely the subgroup generated by a Dehn twist around a curve parallel to the boundary. This means that it would not be acylindrically hyperbolic, so many of our tools would not apply, however \Cref{InfFICentre} tells us that it would also not have uniform product set growth.

To deal with subgroups of such groups, we will employ the capping homomorphism, which effectively says that either our subgroup has infinite centre, or we can view it as a subgroup of a mapping class group of a surface without boundary.

\begin{prop}[Capping homomorphism]
	\emph{\cite[Theorem 3.18]{Farb2011}}
	\label{Capping}
	Let $S$ be a surface with boundary components $\{\xi_1,\ldots,\xi_l\}$, and let $S'$ be the surface obtained from $S$ by capping each boundary component by a once punctured disc. Then there exists a natural homomorphism $\psi:MCG(S)\to MCG(S')$, with kernel $\langle T_{\xi_1},\ldots, T_{\xi_l}\rangle$, where $T_{\xi_i}$ is a Dehn twist around $\xi_i$.
\end{prop}

\begin{rem}
	As every element of $MCG(S)$ fixes the boundary components by definition, we have by \Cref{DehnCommute} that $\langle T_{\xi_1},\ldots, T_{\xi_l}\rangle$ is central in $MCG(S)$.
\end{rem}

\subsection{Short loxodromics in right-angled Artin groups}

Recall that the aim of this section is to say something about product set growth in virtual subgroups of right-angled Artin groups and mapping class groups. In the right-angled Artin group case we would like to use \Cref{ShortLoxodromicCor} to prove uniform product set growth for as many subgroups as possible. To do so we need to show that we can quickly generate loxodromic elements in the action of our group on its extension graph. We already have a characterisation of these loxodromic elements, as given by \Cref{LoxChar}, and we also have a result about the quick generation of pseudo-Anosovs in subgroups of mapping class groups from \Cref{ShortPA}.

In this section we will use these facts to get an analogous result to \Cref{ShortPA} for right-angled Artin groups, using the existence of a natural embedding of right-angled Artin groups into mapping class groups \cite{Clay2012}. The idea is that we can embed any subgroup $H$ of a right-angled Artin group into a mapping class group, then use \Cref{ShortPA} to quickly generate a mapping class which has the same active subsurface as the embedded subgroup. This will then correspond to an element in $H$ with the same essential support as $H$. Under certain circumstances this will be a loxodromic element in the original right-angled Artin group.

There are several ways to see right-angled Artin groups as subgroups of mapping class groups, see for example Section 7.3 of \cite{Koberda2012}. The construction we use here is taken from Section 2.4 of \cite{Clay2012}, and for a more detailed picture of it we refer to there.

Given a finite simple graph $\Gamma$, with vertex set $V=\{s_1,\ldots,s_k\}$, there exists a closed surface $S$ and a collection of subsurfaces $\mathbb{X}=\{X_1,\ldots, X_k\}$ such that:
\begin{itemize}
	\item Each $X_i$ is a twice punctured torus.
	\item $X_i\cap X_j=\emptyset$ if and only if $\{s_i,s_j\}$ is an edge in $\Gamma$.
	\item  $X_i\cap X_j\neq\emptyset$ if and only if $X_i$ and $X_j$ cannot be isotoped to be disjoint.
	\item If $X_i\cap X_j\neq\emptyset$ then this intersection is homeomorphic to a disc.
	\item $X_i$ and $X_j$ cannot be isotoped such that $X_i\subset X_j$ when $i\neq j$.
\end{itemize}
A picture of such a surface can be found in Figure 4 of \cite{Clay2012}. For each $X_i$ we then pick an $f_i$ in $MCG(S)$ that is pseudo-Anosov on $X_i$, and the identity elsewhere.

\begin{thm}
	\emph{\cite{Clay2012}}
	\label{Embedding}
	Given the collection $\mathbb{F}=\{f_1,\ldots, f_k\}$, there exists $M\in\mathbb{N}$ such that the map $\phi:A(\Gamma)\to MCG(S)$ defined by $\phi(s_i)=f_i^M$ is an injective homomorphism.
\end{thm}

This statement allows us to see $A(\Gamma)$ as a subgroup of $MCG(S)$. In particular, every element of $A(\Gamma)$ will be sent to a mapping class that is pseudo-Anosov on some collection of connected subsurfaces of $S$, and is the identity elsewhere. These subsurfaces can be constructed explicitly. The following notation is taken directly from \cite{Clay2012}.

\begin{note}
	Let $\{X_1,\ldots,X_r\}\subset\mathbb{X}$. We denote by $\text{Fill}(X_1,\ldots,X_r)$ the minimal union of essential subsurfaces of $S$ that contains $X_1\cup\cdots\cup X_r$. That is, if $X_1\cup\cdots\cup X_r$ has any discs in its complement then we add these discs to get $\text{Fill}(X_1,\ldots,X_r)$.
\end{note}

\begin{rem}
	Note that we can assume $\text{Fill}(X_1,\ldots,X_r)$ is connected if and only if $X_1\cup\cdots\cup X_r$ is connected, and that $\text{Fill}(X_i)\cap\text{Fill}(X_j)= \emptyset$ if and only if $X_i\cap X_j=\emptyset$.
\end{rem}

\begin{note}
	Suppose that $1\neq g\in A(\Gamma)$ is cyclically reduced, and that $r$ is the minimal number such that $g$ is a word in the first $r$ generators $s_1,\ldots,s_r$, changing the indices if necessary. Then we say that $\text{Fill}(g)=\text{Fill}(X_1,\ldots,X_r)$. If $g=hg'h^{-1}$ for some $h,g'\in A(\Gamma)$ then $\text{Fill}(g)=\phi(h)\text{Fill}(g')$. If $g=1$, we say that $\text{Fill}(g)=\emptyset$.
\end{note}


\begin{thm}
	\label{PseudoAnosov}
	\emph{\cite{Clay2012}} For any $g\in A(\Gamma)$ we have that $\phi(g)$ is pseudo-Anosov on each connected component of $\emph{Fill}(g)$, and is the identity elsewhere.
\end{thm}

\begin{rem}
	\label{FillActive}
	In other words, \Cref{PseudoAnosov} tells us that $A(\Gamma)$ can be seen as a pure subgroup of $MCG(S)$, and for any $g\in A(\Gamma)$ the active subsurface of $\phi(g)$ is exactly $\text{Fill}(g)$.
\end{rem}

Given a finite symmetric subset $U$ of a right-angled Artin group, our aim will be to generate $g\in U^n$, for $n$ bounded, such that the essential support for $g$ is the same as for $U$. The first step in finding such an element is \Cref{Support}, the proof of which requires the following terminology.

\begin{defn}
	For a vertex $v$ in a graph $\Gamma$, the \emph{link} of $v$, denoted $\text{link}(v)$, is the set of vertices adjacent to $v$ in $\Gamma$.
\end{defn}

The method used in the following proof was suggested by Ric Wade.

\begin{lem}
	\label{Support}
	Let $\Gamma$ be a finite graph, and let $U\subset A(\Gamma)$ be finite, where $U\neq \{1\}$. Then for every $s\in \emph{esupp}(U)$, there exists $g\in U^2$ such that $s\in \emph{esupp}(g)$.
\end{lem}

\begin{proof}
	As essential supports are invariant under conjugation, without loss of generality we can assume that $U\subset A(\Gamma(\text{esupp}(U)))$. Let $s\in \text{esupp}(U)$. By \Cref{EssentialCyclicallyReduced} and the fact that right-angled Artin groups are torsion free, we can see that if $s\in\text{esupp}(u)$ for some $u\in U$, then $s\in\text{esupp}(u^2)$. We can therefore assume that for all $u\in U$ we have that $s\notin\text{esupp}(u)$.
	
	Let $V'=\text{esupp}(U)\backslash\{s\}$, and note that every $u\in U$ is conjugate into $A(\Gamma(V'))$. Consider $H=A(\Gamma(V'\cap\text{link}(s)))\leqslant A(\Gamma(V'))$, and let $\varphi$ be the identity map on $H$. From the definition of a HNN-extension, we can see that this homomorphism gives us that $A(\Gamma(\text{esupp}(U)))=A(\Gamma(V'))\ast_{H}$. That is, if $A(\Gamma(V'))=\langle V' | R\rangle$ then
	\begin{equation*}
		A(\Gamma(\text{esupp}(U)))=\langle V'\cup\{s\}\ |\ R\cup\{svs^{-1}=v\ |\ v\in V'\cap\text{link}(s)\}\rangle.
	\end{equation*}
	
	Consider the action of $\langle U\rangle\leqslant A(\Gamma(\text{esupp}(U)))$ on the Bass-Serre tree associated with this HNN-extension. As every $u\in U$ is conjugate into $A(\Gamma(V'))$, every $u$ has a fixed point in this action. If the same was true of every $g\in U^2$, then by \cite[p.~64]{Serre1980} we would have that $\langle U\rangle$ has a fixed point in this action, and so $U$ would be conjugate into $A(\Gamma(V'))$. This would contradict the minimality of $\text{esupp}(U)$, therefore for some $g\in U^2$ we have that $s\in \text{esupp}(g)$.
\end{proof}

We also need the following lemma.

\begin{lem}
	\label{Intersection}
	For any $s\in V$ and $g\in A(\Gamma)$, there exists a genus one subsurface $T^{(g)}$ of $\emph{Fill}(s)$ such that $\phi(g)T^{(g)}$ is also a subsurface of $\emph{Fill}(s)$.
\end{lem}

\begin{proof}
	Let $s\in V$, and consider $g\in A(\Gamma)$ written as a reduced word $g=g_1\cdots g_n$, where each $g_i$ is a vertex or its inverse. We will prove our statement by induction on how many of the $g_i$'s are not equal to $s$ or $s^{-1}$, so let $\mu(g)=|\{g_i:g_i\neq s^{\pm 1}\}|$.
	
	We begin with the base case, $\mu(g)=0$. Then as each $g_i= s^{\pm 1}$ we have that $\phi(g_i)\text{Fill}(s)=\text{Fill}(s)$, so $\phi(g)\text{Fill}(s)=\text{Fill}(s)$. We therefore have that $T^{(g)}=\text{Fill}(s)$ is our required subsurface.
	
	Suppose that the statement is true for every $g\in A(\Gamma)$ such that $\mu(g)=k$. Now consider $g\in A(\Gamma)$ such that $\mu(g)=k+1$. Let $g_i$ be the first letter in $g_1\cdots g_n$ such that $g_i\neq s^{\pm 1}$. Then $\mu(g_{i+1}\cdots g_n)=k$, so there exists a genus one subsurface $T^{(g_{i+1}\cdots g_n)}$ of $\text{Fill}(s)$ such that $\phi(g_{i+1}\cdots g_n)T^{(g_{i+1}\cdots g_n)}\subset \text{Fill}(s)$.
	
	Let $Y=\text{Fill}(s)\backslash\bigcup_{s'\in V\backslash\{s\}}\text{Fill}(s')$, so the part of $\text{Fill}(s)$ that is fixed under the action of $\phi(s')$ for any $s'\in V\backslash\{s\}$. Consider $Y\cap\phi(g_{i+1}\cdots g_n)T^{(g_{i+1}\cdots g_n)}$, and note that some connected component must have genus one, as each intersection that we remove is a disc in $\text{Fill}(s)$ that can be isotoped to be disjoint to any genus one subsurface. Let $T$ be this component. We then have that $\phi(g_i)T=T$, and hence $\phi(g_1\cdots g_i)T\subset \text{Fill}(s)$, as for each $j<i$ we have that $g_j=s^{\pm 1}$.
	
	Let $T^{(g)}=\phi(g_{i+1}\cdots g_n)^{-1} T$. As $T\subset \phi(g_{i+1}\cdots g_n)T^{(g_{i+1}\cdots g_n)}$ we get that $T^{(g)}\subset T^{(g_{i+1}\cdots g_n)}\subset \text{Fill}(s)$. We also have that $\phi(g)T^{(g)}=\phi(g_1\cdots g_i)T\subset \text{Fill}(s)$, so this is our required genus one subsurface. This concludes our induction, and as a consequence the statement is proved.
\end{proof}

We can now say something about the genus of $\text{Fill}(g)$, depending on the number of vertices in its essential support.

\begin{prop}
	\label{SameGenus}
	Let $\Gamma$ be a finite graph. For every $U\subset A(\Gamma)$, and $g\in\langle U\rangle$, if $\emph{Fill}(g)$ has the same genus as the active subsurface of $\phi(\langle U\rangle)$, then $\emph{esupp}(g)=\emph{esupp}(U)$.
\end{prop}

\begin{proof}
	Suppose $\text{esupp}(g)\neq \text{esupp}(U)$. As $\langle U\rangle$ is conjugate into $A(\Gamma (\text{esupp}(U)))$, so is $g$, so $\text{esupp}(g)\subsetneq \text{esupp}(U)$. We want to show that $\text{Fill}(g)$ has strictly smaller genus than the active subsurface of $\phi(\langle U\rangle)$. As $\text{Fill}(g)$ is a subsurface of the active subsurface of $\phi(\langle U\rangle)$, it suffices to show that there is a genus one subsurface of the active subsurface of $\phi(\langle U\rangle)$ that is not contained in $\text{Fill}(g)$.
	
	Let $s\in \text{esupp}(U)\backslash \text{esupp}(g)$. \Cref{Support} tells us that there exists $h\in \langle U\rangle$ such that $s\in \text{esupp}(h)$. Suppose $h=h'a(h')^{-1}$ as a reduced word, and that $a$ is cyclically reduced. Then $s$ or $s^{-1}$ is a letter of $a$, and $\phi(h')\text{Fill}(s)$ is a subsurface of $\phi(h')\text{Fill}(a)=\text{Fill}(h)$, which is a subsurface of the active subsurface of $\phi(\langle U\rangle)$.
	
	It is therefore sufficient to show that there is a genus one subsurface of $\phi(h')\text{Fill}(s)$ which is disjoint from $\text{Fill}(g)$. Suppose $g=g'b(g')^{-1}$ as a reduced word, and that $b$ is cyclically reduced. Note that neither $s$ nor $s^{-1}$ is a letter of $b$. We have that $\text{Fill}(g)=\phi(g')\text{Fill}(b)$, and we can note that any genus one subsurface disjoint from $\bigcup_{s'\in \text{esupp}(g)}\text{Fill}(s')$ can also be isotoped to be disjoint from $\text{Fill}(b)$.
	
	By \Cref{Intersection}, there exists a genus one subsurface $T$ of $\text{Fill}(s)$ such that $\phi((g')^{-1}h')T$ is also a subsurface of $\text{Fill}(s)$. Let $Y=\text{Fill}(s)\backslash\bigcup_{s'\in V\backslash\{s\}}\text{Fill}(s')$, so the part of $\text{Fill}(s)$ that is fixed under the action of $\phi(s')$ for any $s'\in V\backslash\{s\}$. Consider $Y\cap\phi((g')^{-1}h')T$, and note that some connected component must have genus one, as each intersection that we remove is a disc in $\text{Fill}(s)$ that can be isotoped to be disjoint to any genus one subsurface. Let $T'$ be this component, and let $T''=\phi((h')^{-1}g')T'$.
	
	We then have that $T''\subset T\subset \text{Fill}(s)$, and $\phi((g')^{-1}h')T''=T'$, so $\phi((g')^{-1}h')T''$ is a subset of $ \text{Fill}(s)\backslash\bigcup_{s'\in V\backslash\{s\}}\text{Fill}(s')$, and so is disjoint from $\text{Fill}(b)$. This means that $\phi(g')\phi((g')^{-1}h')T''=\phi(h')T''$, which is a genus one subsurface of $\phi(h')\text{Fill}(s)$, is disjoint from $\phi(g')\text{Fill}(b)=\text{Fill}(g)$, so we are done.
\end{proof}

This allows us to get an analogous result to \Cref{ShortPA} for right-angled Artin groups.

\begin{thm}
	\label{FullSupport}
	Let $\Gamma$ be a finite graph. There exists a constant  $N=N(\Gamma)\in\mathbb{N}$ such that for every finite symmetric $U\subset A(\Gamma)$, there exists $n\leqslant N$ and $g\in U^n$ such that $\emph{esupp}(g)=\emph{esupp}(U)$.
\end{thm}


\begin{proof}
	Let $S$ be the surface constructed from $\Gamma$ in \Cref{Embedding}, and note that $S$ is non-sporadic and without boundary. Let $N$ be the constant from \Cref{ShortPA}.
	
	Now take any finite symmetric $U\subset A(\Gamma)$, and consider $\phi(U)\subset MCG(S)$, as given by \Cref{Embedding}. By \Cref{ShortPA} there exists $n\leqslant N$ and $f\in \phi(U)^n=\phi(U^n)$ such that $f$ has the same active subsurface as $\langle \phi(U)\rangle=\phi(\langle U\rangle)$.
	
	As $\phi$ is injective, there exists $g\in U^n$ such that $\phi(g)=f$. Note that, by \Cref{PseudoAnosov}, the active subsurface of $\phi(g)$ is exactly $\text{Fill}(g)$. We can therefore apply \Cref{SameGenus} to conclude that $\text{esupp}(g)=\text{esupp}(U)$.
\end{proof}

Finding elements of full essential support means that under certain circumstances we can quickly generate loxodromic elements in $A(\Gamma)$. The following corollary is obtained by combining \Cref{FullSupport} with \Cref{LoxChar}.

\begin{cor}
	\label{ExtensionShortLoxPrime}
	Let $\Gamma$ be a finite graph. There exists a constant $N=N(\Gamma)\in\mathbb{N}$ such that for every finite symmetric $U\subset A(\Gamma)$, where $\Gamma(\emph{esupp}(U))$ is connected, and is neither an isolated vertex nor a join, there exists $n\leqslant N$ such that $U^n$ contains a loxodromic element on $\Gamma(\emph{esupp}(U))^e$.
\end{cor}

As a consequence, we also get the following, more algebraic, statement.

\begin{cor}
	\label{ExtensionShortLox}
	Let $\Gamma$ be a finite graph. There exists a constant $N=N(\Gamma)\in\mathbb{N}$ such that for every finite symmetric $U\subset A(\Gamma)$, where $\Gamma(\emph{esupp}(U))$ is connected and $\langle U\rangle$ is neither cyclic nor contained non-trivially in a direct product, there exists $n\leqslant N$ such that $U^n$ contains a loxodromic element on $\Gamma(\emph{esupp}(U))^e$.
\end{cor}


As mentioned in the introduction, the ability to quickly generate these loxodromic elements has an application to $\xi(G)=\{\omega(G,S)\ :\ S\text{ is a finite generating set of }G\}$, the set of exponential growth rates of $G$ with respect to its finite generating sets. In particular, the fact that loxodromic elements can be used to link the growth of a set with its size is a key step in the proof of the following theorem.

\begin{thm}
	\label{WellOrdered}
	\emph{\cite{Fujiwara2021}}
	Let $G$ be a group with an acylindrical and non-elementary action on a hyperbolic graph $X$. Suppose that there exists a constant $k\in\mathbb{N}$ such that for any finite symmetric generating set $S$ of $G$, the product set $S^k$ contains a loxodromic element on $X$. Assume that $G$ is equationally Noetherian. Then, $\xi(G)$ is a well-ordered set.
\end{thm}

We can combine this with \Cref{ExtensionShortLox} to obtain the following. Note that it is a standard fact that if $G=G_1\times\cdots\times G_m$ is a direct product of finitely generated groups, then $\sup_{1\leqslant i\leqslant m}\omega(G_i)\leqslant \omega(G)\leqslant \omega(G_1)\cdots \omega(G_m)$.

\begin{thm}
	\label{WellOrderedCor}
	Let $\Gamma$ be a finite graph. Suppose that $G\leqslant A(\Gamma)$ is finitely generated, and is not contained non-trivially in a direct product where the projection of $G$ to more than one factor has exponential growth. Then $\xi(G)$ is a well-ordered set.
\end{thm}

\begin{proof}
	We begin by noting that there is a bound on the number of factors of any direct product subgroup of a right-angled Artin group, as a consequence of \Cref{FreeAbelian} and the fact that right-angled Artin groups are subgroups of mapping class groups \cite{Clay2012}.
	
	If $G$ does not have exponential growth, then $\xi(G)=\{1\}$, which is trivially well-ordered. We can therefore assume that $G$ is not cyclic, and we can also assume that $G$ is not contained non-trivially in a direct product where none of the factors have exponential growth.
	
	Now suppose $G$ is contained non-trivially in a direct product, and the projection of $G$ to exactly one factor has exponential growth. Then we have that the set of growth rates of $G$ is the same as the set of growth rates of the projection of $G$ to this factor. We can therefore assume that $G$ is not contained non-trivially in a direct product, as the only remaining case is the one ruled out in the statement of the theorem.
	
	Note that $G$ is a subgroup of a right-angled Artin group, so is linear \cite{Hsu1999}, and therefore is equationally Noetherian \cite{Baumslag1999}. If $\Gamma(\text{esupp}(G))$ is connected, then we conclude using \Cref{ExtensionShortLox} and \Cref{WellOrdered}. If $\Gamma(\text{esupp}(G))$ is not connected, we instead use its action on the associated Bass-Serre tree (see Example 1.7 in \cite{Fujiwara2021}), and again conclude using \Cref{WellOrdered}.
\end{proof}

\subsection{Product set growth in right-angled Artin groups}

We can now use the results of the previous section to answer \Cref{MainQuestion} for right-angled Artin groups, and their virtual subgroups. For certain finite symmetric sets $U\subset A(\Gamma)$, we will use either \Cref{ExtensionShortLox}, or the action of a free product on its Bass-Serre tree, to quickly generate a loxodromic element in the action of $\langle U\rangle$ on one of a finite collection of quasi-trees. This will allow us to apply \Cref{ShortLoxodromicCor}. We will then be able to extend this using some of our observations about direct products from Section 2.2.

When working with direct products, the constants that we obtain depend on the number of factors in the direct product. Fortunately for us, it is a well-known fact that the maximal rank of a free abelian subgroup of a right-angled Artin group $A(\Gamma)$ is equal to the number of vertices in a maximal complete subgraph of $\Gamma$, and that this also bounds the number of factors in a non-trivial direct product subgroup. See, for example \cite{Charney2009,Kim2014a}, or Section 5.3 of \cite{Kerr2022} for a proof. As mentioned previously, the maximal rank being bounded is also a trivial consequence of \Cref{FreeAbelian} and the fact that right-angled Artin groups are subgroups of mapping class groups, however it is useful to have a specific bound linked directly to the properties of $\Gamma$.

%
%

\begin{rem}
	When applying \Cref{ShortLoxodromicCor} in the following, we assume that various subgroups are not $\mathbb{Z}$, rather than not virtually $\mathbb{Z}$. This is possible as $A(\Gamma)$ is torsion-free, so as any torsion-free virtually free group is free \cite{Stallings1968}, we have that any virtually $\mathbb{Z}$ subgroup of $A(\Gamma)$ is in fact isomorphic to $\mathbb{Z}$.
\end{rem}

\begin{thm}
	\label{RaagGrowth2}
	Let $\Gamma$ be a finite graph. There exist constants $\alpha,\beta > 0$ such that for every finite symmetric $U \subset A(\Gamma)$, at least one of the following must hold:
	\begin{enumerate}
		\item The centre of $\langle U\rangle$ is non-trivial.
		\item $|U^n| \geqslant (\alpha|U|)^{\beta n}$ for every $n \in \mathbb{N}$.
	\end{enumerate}
\end{thm}

\begin{proof}
	Let $U\subset A(\Gamma)$ be finite and symmetric, such that $\langle U\rangle$ has trivial centre. Suppose that $\langle U\rangle$ is finite. As right-angled Artin groups are torsion-free, this means that $\langle U\rangle$ is trivial. Therefore, so long as we ensure that $\alpha\leqslant 1$, we will have that $U$ satisfies $|U^n| \geqslant (\alpha|U|)^{\beta n}$ for every $n \in \mathbb{N}$.
	
	We now assume that $\langle U\rangle$ is infinite. We can say that $\langle U\rangle$ is contained non-trivially in a direct product $G_1\times\cdots\times G_m\leqslant A(\Gamma)$, where we allow for the possibility that $m=1$, and let $U_i$ be the projection of $U$ to $G_i$. We note that each $U_i$ is finite and symmetric.
	
	Suppose that some $\langle U_i\rangle \leqslant G_i$ is contained non-trivially in a direct product $G_{i,1}\times G_{i,2}\leqslant G_i$. Then $\langle U\rangle$ is contained in $G_1\times\cdots\times G_{i-1}\times G_{i,1}\times G_{i,2}\times G_{i+1}\times\cdots\times G_m$. We know that there exists $N=N(\Gamma)\in\mathbb{N}$ that bounds the number of factors in a direct product in $A(\Gamma)$, so we can assume that each $\langle U_i\rangle$ is not contained non-trivially in a direct product, and that $m\leqslant N$.
	
	Suppose that some $\langle U_i\rangle$ is $\mathbb{Z}$. Then, by \Cref{ZInjective}, the projection of $\langle U\rangle$ to $G_1\times\cdots \times G_{i-1}\times G_{i+1}\times\cdots \times G_m$ is injective, as otherwise $\langle U\rangle$ would have infinite centre. Hence $\langle U\rangle$ can instead be viewed as a subgroup of the direct product of the remaining factors. We can therefore assume that no $\langle U_i\rangle$ is $\mathbb{Z}$, without loss of generality.
	
	We will now show that, for each $i\in\{1,\ldots, m\}$, the set $U_i$ satisfies the required product set growth inequality. Note that, as product set growth is invariant under conjugation, we can assume that $U_i\subset A(\Gamma(\text{esupp}(U_i)))$. There are now two cases for us to consider.
	
	The first case is that $\Gamma(\text{esupp}(U_i))$ is connected, which means that we can apply \Cref{ExtensionShortLox}. Let $N'=N'(\Gamma)\in\mathbb{N}$ be the constant that we get from this corollary, then there exists $k\leqslant N'$ such that $U_i^k$ contains a loxodromic in the action on $\Gamma(\text{esupp}(U_i))^e$, which is a quasi-tree. By \Cref{ShortLoxodromicCor}, there exist $\alpha_i,\beta_i>0$, determined by $\text{esupp}(U_i)$ and $N'$, such that $|U_i^n| \geqslant (\alpha_i|U_i|)^{\beta_i n}$ for every $n \in \mathbb{N}$.
	
	The second case is that $\Gamma(\text{esupp}(U_i))$ is disconnected. Write $\Gamma(\text{esupp}(U_i))=\Gamma_1\sqcup \Gamma_2$. Then $A(\Gamma(\text{esupp}(U_i)))=A(\Gamma_1)\ast A(\Gamma_2)$, and we recall the standard fact that any free product acts acylindrically on its associated Bass-Serre tree.
	
	If $\langle U_i\rangle$ fixes a vertex in this tree, then $\langle U_i\rangle$ is contained in the stabiliser of that vertex. This implies $U_i$ is conjugate into $A(\Gamma_1)$ or $A(\Gamma_2)$, which contradicts the minimality of $\text{esupp}(U_i)$. Therefore $\langle U_i\rangle$ does not fix a vertex, so $U_i^2$ must contain a loxodromic element in this action by \cite[p.~64]{Serre1980}. By \Cref{ShortLoxodromicCor}, we again get that there exist $\alpha_i,\beta_i>0$, determined by $\text{esupp}(U_i)$, such that $|U_i^n| \geqslant (\alpha_i|U_i|)^{\beta_i n}$ for every $n \in \mathbb{N}$.
	
	As $\Gamma$ is finite, there are only finitely many possibilities for $\text{esupp}(U_i)$. We can therefore take $\alpha=\inf{\alpha_{i}}$ and $\beta=\inf{\beta_{i}}$ over all possibilities for $\text{esupp}(U_i)$, and we will have that $\alpha,\beta>0$, with the property that $|U_i^n| \geqslant (\alpha|U_i|)^{\beta n}$ for every $n \in \mathbb{N}$ and $i\in\{1,\ldots, m\}$.
	
	We can now apply \Cref{Factors1} to get that
	\begin{equation*}
		|U^n|\geqslant (\alpha|U|^{\frac{1}{m}})^{\beta n}\geqslant(\alpha|U|^{\frac{1}{N}})^{\beta n}=(\alpha^N|U|)^{\frac{\beta n}{N}}.
	\end{equation*}
	Let $\alpha'=\min\{\alpha,\alpha^N,1\}$ and $\beta'=\frac{\beta}{N}$. Combining the cases, we have shown that for every finite symmetric
	$U \subset A(\Gamma)$, where $\langle U \rangle$ has trivial centre, we have that $|U^n| \geqslant (\alpha'|U|)^{\beta' n}$ for all $n \in \mathbb{N}$.
\end{proof}

We note here that, given that \Cref{InfFICentre} tells us that any group with an infinite order element in its centre cannot have uniform product set growth, and the fact that right-angled Artin groups are torsion free, this is the best result we can hope for in this case. We have therefore answered the symmetric version of \Cref{MainQuestion} for right-angled Artin groups.

We can also combine \Cref{RaagGrowth2} with \Cref{Supergroup} to get the following. This gives a dichotomy of subgroups for all virtual subgroups of right-angled Artin groups, which in particular includes finitely generated virtually special groups \cite{Haglund2008}, such as finitely generated Coxeter groups \cite{Haglund2010}.

\begin{cor}
	\label{VirtualRaag}
	Let $\Gamma$ be a finite graph, and let $G$ be a group which virtually embeds into $A(\Gamma)$. There exist constants $\alpha,\beta > 0$ such that for every finite symmetric $U \subset G$, at least one of the following must hold:
	\begin{enumerate}
		\item $\langle U\rangle$ has a finite index subgroup with infinite centre.
		\item $|U^n| \geqslant (\alpha|U|)^{\beta n}$ for every $n \in \mathbb{N}$.
	\end{enumerate}
\end{cor}

\begin{proof}
	Let $H$ be a fixed finite index subgroup of $G$, where $[G:H]=d$, and $H$ embeds as a subgroup of $A(\Gamma)$. Let $U\subset G$ be finite and symmetric. Then $\langle U\rangle \cap H$ is a torsion-free finite index subgroup of $\langle U\rangle$, where $[\langle U\rangle:\langle U\rangle \cap H]\leqslant d$. Suppose the first case does not hold. By this assumption, $\langle U\rangle \cap H$ has trivial centre.
	
	As $\langle U\rangle \cap H$ is a finite index subgroup of a finitely generated group, it is also finitely generated. Let $V$ be a finite symmetric generating set of $\langle U\rangle \cap H$. Then by \Cref{RaagGrowth2} there exist $\alpha,\beta>0$, only dependent on $\Gamma$, such that $|V^n|\geqslant (\alpha|V|)^{\beta n}$ for every $n \in \mathbb{N}$. Let $m=2d-1$, then \Cref{Supergroup} tells us that
	\begin{equation*}
		|U^n|\geqslant \bigg(\frac{\alpha}{2^{\frac{m}{\beta}}d}|U|\bigg)^{\frac{\beta n}{m}}
	\end{equation*}
	for every $n\in\mathbb{N}$.
\end{proof}

We note here that the symmetry requirement in these results come from the use of \Cref{ShortPA} in the proof of \Cref{FullSupport}. If a more direct proof for \Cref{FullSupport} could be found, without requiring symmetry, then we could extend the above results to the non-symmetric case.

\begin{ques}
	Is \Cref{VirtualRaag} true for non-symmetric sets?
\end{ques}

\subsection{Product set growth in mapping class groups}

We now move on to the more general case of subgroups of mapping class groups. We begin by considering only the pure subgroups, which allows us to get a product set growth result that we can then extend to all virtual subgroups.

The general structure of the proof for pure subgroups is the same as in the previous section. That is, we consider our subgroup as a subgroup of a direct product. With the exception of certain cases, we show that each of the factor groups in the direct product satisfy uniform product set growth. We can then apply \Cref{Factors1} to get growth for the whole subgroup.

\begin{thm}
	\label{PureCase}
	Let $G\leqslant MCG(S)$ be a pure subgroup of a mapping class group. There exist $\alpha,\beta>0$ such that for every finite symmetric $U\subset G$, at least one of the following must hold:
	\begin{enumerate}
		\item The centre of $\langle U\rangle$ is non-trivial.
		\item $|U^n| \geqslant (\alpha|U|)^{\beta n}$ for every $n \in \mathbb{N}$. 
	\end{enumerate}
\end{thm}

\begin{proof}
	Suppose $U\subset G$ is symmetric and finite, and that the centre $Z(\langle U\rangle)$ is trivial. Suppose that $\langle U\rangle$ is finite. As pure subgroups are torsion-free, this means that $\langle U\rangle$ is trivial. Therefore, so long as we ensure that $\alpha\leqslant 1$, we will have that $U$ satisfies $|U^n| \geqslant (\alpha|U|)^{\beta n}$ for every $n \in \mathbb{N}$.
	
	We can now assume that $\langle U\rangle$ is infinite. Let $\{\xi_1,\ldots,\xi_l\}$ be the (possibly empty) set of curves in the boundary of $S$, and recall that $T_{\xi_i}$ is the Dehn twist about $\xi_i$. Let $S'$ be the surface obtained from $S$ by capping each boundary component by a once punctured disc. By \Cref{Capping} we have a natural homomorphism $\psi: \langle U\rangle\to MCG(S')$, with kernel $\langle T_{\xi_1},\ldots, T_{\xi_l}\rangle\cap \langle U\rangle$.
	
	Recall from \Cref{DehnCommute} that the group $\langle T_{\xi_1},\ldots, T_{\xi_l}\rangle\cap \langle U\rangle$ is central in $\langle U\rangle$, and therefore by our initial assumption $\langle T_{\xi_1},\ldots, T_{\xi_l}\rangle\cap \langle U\rangle$ is trivial. This means that $\psi$ is injective, and we can therefore suppose without loss of generality that $S$ has no boundary.
	
	Now let $\{\gamma_1,\ldots,\gamma_k\}$ be the (possibly empty) set of curves in the canonical reduction multicurve for $\langle U\rangle$. Let $\mathcal{S}$ be the set of connected components of $S\backslash\bigcup_{i=1}^k\gamma_i$. By \Cref{PureCutting} we have a natural homomorphism $\varphi: \langle U\rangle\to \Pi_{\Sigma\in \mathcal{S}}MCG(\Sigma)$, with kernel $\langle T_{\gamma_1},\ldots, T_{\gamma_k}\rangle\cap \langle U\rangle$.
	
	Recall from \Cref{DehnCommute} that the group $\langle T_{\gamma_1},\ldots, T_{\gamma_k}\rangle\cap \langle U\rangle$ is central in $\langle U\rangle$, and therefore by our initial assumption $\langle T_{\gamma_1},\ldots, T_{\gamma_k}\rangle\cap \langle U\rangle$ is trivial. This means that $\varphi$ is injective, and it therefore makes sense to talk about the projection of $\langle U\rangle$ to a factor of $\Pi_{\Sigma\in \mathcal{S}}MCG(\Sigma)$.
	
	Consider the projection of $\langle U\rangle$ to such a factor $MCG(\Sigma)$. By \Cref{PureCutting}, this projection is either trivial, or contains a pseudo-Anosov on $\Sigma$. If the projection is trivial, we can ignore this factor, and consider $\langle U\rangle$ as a subgroup of the direct product of the remaining factors. If the projection is $\mathbb{Z}$, then we can similarly ignore this factor, as \Cref{ZInjective} tells us that either the projection to the remaining factors is injective, or $\langle U\rangle$ has infinite centre.
	
	Let $\mathcal{S}'$ be the surfaces $\Sigma$ in $\mathcal{S}$ such that the projection of $\langle U\rangle$ to $MCG(\Sigma)$ is neither trivial nor $\mathbb{Z}$. The above reasoning tells us that we can see $\langle U\rangle$ as a subgroup of $\Pi_{\Sigma\in \mathcal{S}'}MCG(\Sigma)$. Let $U_{\Sigma}$ be the projection of $U$ to $MCG(\Sigma)$ for some $\Sigma\in \mathcal{S}'$. As $\langle U_{\Sigma}\rangle $ contains a pseudo-Anosov on $\Sigma$, \Cref{ShortPA1} tells us that there is some $N\in\mathbb{N}$, dependent only on $S$, such that for some $k\leqslant N$ we have that $U_{\Sigma}^k$ contains a pseudo-Anosov on $\Sigma$.
	
	We first suppose that $\Sigma$ is non-sporadic. As $\langle U_{\Sigma}\rangle $ contains a pseudo-Anosov on $\Sigma$, \Cref{ShortPA1} tells us that there is some $N\in\mathbb{N}$, dependent only on $S$, such that for some $k\leqslant N$ we have that $U_{\Sigma}^k$ contains a pseudo-Anosov on $\Sigma$. Recall that the action of $\langle U_{\Sigma}\rangle$ on the curve complex associated to $\Sigma$ is acylindrical, and this curve complex is hyperbolic. Therefore, as $\langle U_{\Sigma}\rangle $ is not $\mathbb{Z}$, \Cref{HypLoxCor} tells us that there exist $\alpha_{\Sigma},\beta_{\Sigma}>0$, determined by $\Sigma$ and $N$, such that $|U_{\Sigma}^n| \geqslant (\alpha_{\Sigma}|U_{\Sigma}|)^{\beta_{\Sigma} n}$ for every $n \in \mathbb{N}$.
	
	Alternatively, if $\Sigma$ is sporadic, then $MCG(\Sigma)$ is hyperbolic. Therefore, as $\langle U_{\Sigma}\rangle $ is not $\mathbb{Z}$, Theorem 1.1 in \cite{Delzant2020} tells us that again there exist $\alpha_{\Sigma},\beta_{\Sigma}>0$ determined by $\Sigma$ such that $|U_{\Sigma}^n| \geqslant (\alpha_{\Sigma}|U_{\Sigma}|)^{\beta_{\Sigma} n}$ for every $n \in \mathbb{N}$.
	
	There are only finitely many homeomorphism classes of essential subsurfaces of $S$, that is the subsurfaces with boundary components that are either in $\partial S$ or are essential in $S$, so we can take $\alpha=\inf{\alpha_{\Sigma}}$ and $\beta=\inf{\beta_{\Sigma}}$ over all possible subsurfaces, and we will get that $\alpha,\beta>0$.
	
	Now let $g$ be the genus of $S$, $c$ be the number of connected components, and recall that if our original surface had $b$ boundary components and $p$ punctures, then by our use of \Cref{Capping} the surface that we are now considering has $b+p$ punctures. By \Cref{FreeAbelian}, the maximal rank of a free abelian subgroup of $MCG(S)$ is $3g+b+p-3c$. Therefore as each factor of $\Pi_{\Sigma\in \mathcal{S}'}MCG(\Sigma)$ contains an infinite order element, we must have that $|\mathcal{S}'|\leqslant 3g+b+p-3c$. Let
	\begin{equation*}
		\alpha'=\min\{\alpha,\alpha^{3g+b+p-3c}\}
	\end{equation*}
		and let 
		\begin{equation*}
			\beta'=\frac{\beta}{3g+b+p-3c}.
		\end{equation*}
		Then by \Cref{Factors1} we have that $|U^n| \geqslant (\alpha'|U|)^{\beta' n}$ for every $n \in \mathbb{N}$.
		
		Combining this with the case that $\langle U\rangle$ is finite, we see that if we let $\alpha''=\min\{\alpha',1\}$, then any set $U$ that satisfies our initial hypotheses will also satisfy $|U^n| \geqslant (\alpha''|U|)^{\beta' n}$ for every $n \in \mathbb{N}$.
	\end{proof}
	
	By \Cref{InfFICentre}, this is a dichotomy of subgroups, and so answers the symmetric case of \Cref{MainQuestion} for the pure subgroups of mapping class groups.
	
	\Cref{FiniteIndex} tells us that every subgroup of a mapping class group has a pure subgroup of finite index. As with right-angled Artin groups, we can therefore combine \Cref{PureCase} with \Cref{Supergroup} to get the following.
	
	\begin{cor}
		\label{MainResult}
		Let $MCG(S)$ be a mapping class group, and let $G$ be a group which virtually embeds into $MCG(S)$. There exist $\alpha,\beta>0$ such that for every finite symmetric $U\subset G$, at least one of the following must hold:
		\begin{enumerate}
			\item $\langle U\rangle$ has a finite index subgroup with infinite centre.
			\item $|U^n| \geqslant (\alpha|U|)^{\beta n}$ for every $n \in \mathbb{N}$. 
		\end{enumerate}
	\end{cor}
	
	\begin{proof}
		Let $H$ be a fixed finite index subgroup of $G$, such that $H$ embeds as a subgroup of $MCG(S)$. Let $P$ be a fixed finite index pure subgroup of $MCG(S)$. Then $H\cap P$ is a finite index subgroup of $G$, as well as a pure subgroup of $MCG(S)$. The rest of the proof follows the same format as \Cref{VirtualRaag}.
	\end{proof}
	
	This is a dichotomy of subgroups by \Cref{InfFICentre}, so answers the symmetric case of \Cref{MainQuestion} for all of the virtual subgroups of mapping class groups.
	
	\begin{rem}
		As right-angled Artin groups embed as pure subgroups of mapping class groups, \Cref{PureCase} reproves \Cref{RaagGrowth2}, and \Cref{MainResult} reproves \Cref{VirtualRaag}.
	\end{rem}

\bibliographystyle{alpha}
\bibliography{C:/Users/ak13860/OneDri\string~1/Documents/Notesa\string~1/references.bib}

\end{document}